\documentclass[smallextended,numbook,runningheads]{svjour3}
\usepackage{amsmath}
\usepackage{lmodern}
\smartqed  
\usepackage{url}

\usepackage{amsfonts,amssymb}
\usepackage{epsfig}
\usepackage{booktabs}

\usepackage{xcolor}
\usepackage{mathrsfs}
\graphicspath{{figure/}}
\usepackage{float}
\usepackage{nicematrix}
\usepackage{mathabx}

\usepackage{algorithm}
\usepackage{algpseudocode}

\usepackage[T3,T1]{fontenc}

\journalname{}
\date{ \phantom{b} \vspace{45mm}\phantom{e}}
\def\makeheadbox{\hspace{-4mm}Version of 16 June 2026}

\def\R{\mathbb{R}}
 \def\iu{\mathrm{i}}
\def\eps{\varepsilon}
\def\d{\mathrm{d}}
\def\e{{\mathrm e}}

\def\wt{\widetilde}
\def\wh{\widehat}

\newdimen\GGGlength
\newdimen\GGGheight
\newbox\GGGbox
\def\GGGput[#1,#2](#3,#4)#5{%
  \setbox\GGGbox\vbox{\hbox{#5}\kern0pt}%
  \GGGlength\wd\GGGbox%
  \divide\GGGlength by100 \multiply\GGGlength by#1%
  \GGGheight\ht\GGGbox%
  \divide\GGGheight by100 \multiply\GGGheight by#2%
  \put(#3,#4){\kern-\GGGlength\raise-\GGGheight\box\GGGbox}}

\def\black{\color{black}}

\definecolor{mygreen}{RGB}{10,120,10}

\def\bcl{\black}
\def\ecl{\black}

\def\bys{\black}
\def\eys{\black}

\begin{document}

\title{Weighted finite difference methods for a nonlinear Klein--Gordon equation with high oscillations in space and time}

\titlerunning{Weighted finite difference methods for nonlinear Klein--Gordon equations}

\author{Yanyan~Shi$^{1}$, Christian Lubich$^1$}
\authorrunning{Y.\ Shi, Ch.\ Lubich}

\institute{
$^1$~Mathematisches Institut, Univ.\ T\"ubingen, D-72076 T\"ubingen, Germany.\\
\phantom{$^2$~}\email{\{Shi, Lubich\}@na.uni-tuebingen.de}
}

\date{ }

\maketitle

\begin{abstract} 
We consider a nonlinear Klein--Gordon equation in the nonrelativistic limit regime with initial data in the form of a modulated highly oscillatory exponential. In this regime of a small scaling parameter $\eps\ll 1$, the solution exhibits rapid oscillations in both time and space. 
\bcl  
The solution is approximated, up to $\mathcal{O}(\eps)$, by a superposition of two polarized solutions, which are wave packets that move with opposite group velocities proportional to $\eps^{-1}$.  The equations for polarized solutions are formulated in co-moving coordinates and are then discretized \ecl by an explicit and an implicit exponentially weighted finite difference method. While the explicit weighted leapfrog method needs to satisfy a CFL-type stability condition, the implicit weighted Crank--Nicolson method is unconditionally stable. 
Both methods achieve second-order accuracy with time steps and mesh sizes that are not restricted in magnitude by~$\eps$. \bcl For the approximation of polarized solutions, \ecl the methods are uniformly convergent in the range from arbitrarily small to moderately bounded $\eps$. 
Numerical experiments illustrate the theoretical results.

\bigskip


\bigskip


\end{abstract}

\section{Introduction}

\subsection{Problem formulation}
We consider the numerical solution of nonlinear dispersive wave equations that exhibit highly oscillatory behaviour in both time and space. As a concrete example, we focus on the nonlinear Klein--Gordon equation in the scaling  known as the nonrelativistic limit regime; see e.g.~\cite{bao2012analysis,bao2014uniformly,faou2014asymptotic,baumstark2018uniformly}. With a small parameter $0<\eps\ll 1$, we consider
\begin{equation}\label{eq:KG}
\eps^2\partial^2_t u - \partial^2_{x} u + \frac{1}{\eps^2}u + \lambda |u|^2 u = 0, \qquad x\in \mathbb{R}, \ 0\le t \le T,
\end{equation}
where the solution $u = u(t,x)$ is complex-valued. Here, $\lambda$ is a fixed nonzero real number independent of $\eps$, and
$T>0$ is a final time independent of $\eps$. The initial conditions are given by modulated highly oscillatory exponentials,
\begin{equation}\label{eq:init}
u(0,x) = a_0(x)\e^{\iu \kappa  x / \eps},\quad \partial_t u(0,x) = \frac{1}{\eps^2}b_0(x)\e^{\iu \kappa  x / \eps}.
\end{equation}
 The initial conditions include a rapidly oscillating phase factor with a fixed wave number $\kappa \in \mathbb{R} \setminus \{0\}$. The functions $a_0, b_0 : \mathbb{R} \to \mathbb{C}$ are smooth profiles independent of $\eps$ having bounded support. 
 
 Our objective is to numerically approximate the solution over a fixed time interval $0 \le t \leq T$ that is independent of $\eps$.

\subsection{Related literature}
Numerically solving \eqref{eq:KG}--\eqref{eq:init} is challenging due to the high oscillations of the solution in both time and space, \bcl and further because of the two opposite fast group velocities, inversely proportional to $\eps$, of polarized wave packets.  \ecl The literature predominantly focuses on smooth initial data and handling temporal oscillations, with methods including finite difference schemes~\cite{bao2012analysis}, splitting techniques~\cite{dong2014time}, approaches based on Duhamel’s formula~\cite{baumstark2018uniformly}, time two-scale methods~\cite{chartier2015uniformly,bao2019comparison}, and multiscale expansions~\cite{faou2014asymptotic,bao2014uniformly}.
However, only few numerical works address simultaneous spatial and temporal oscillations for \eqref{eq:KG} or related equations. We are aware of two approaches so far: (i) the nonlinear geometric optics method of~\cite{crouseilles2017nonlinear}, where a \bcl differently scaled class of wave-type partial differential equations with a significantly wider class of highly oscillatory initial data (of WKB type) is addressed using elaborate numerical and analytical multiscale techniques, \ecl and (ii) the combination of analytical approximations with tailor-made time integration in~\cite{jahnke2024analytical} for a wider class of partial differential equations with highly oscillatory initial data that are real parts of \eqref{eq:init}. \bcl We take a different approach in this paper, which is algorithmically simpler than (i) and (ii) but is devised and studied for the specific problem \eqref{eq:KG}--\eqref{eq:init} for concreteness.

\subsection{Weighted finite differences}
We build on the approach in our previous paper~\cite{shi2025weighted}, where weighted finite difference methods for a nonlinear Schr\"odinger equation with highly oscillatory solutions in space and time were introduced and analyzed. 
\bcl
In this subsection we describe common ground and differences between the previous paper~\cite{shi2025weighted} and the current paper, where numerical methods of the same type are applied to a different equation that poses substantially different challenges. One purpose of the paper is precisely to show that the weighted finite difference approach is not limited to the semiclassical nonlinear Schr\"odinger equation, but can be adapted to much larger classes of dispersive problems, of which the nonlinear Klein--Gordon equation studied here is a prominent example. While the adaptation of the algorithm is done easily (which we consider a big plus), the reformulation of the problem, which is necessary before applying the numerical discretization, and the error analysis pose new challenges that we address in this paper.

Weighted finite difference methods are conceptually simple, are as easy to implement as classical finite difference methods, \bcl and have essentially the same computational cost per time step when using the same spatial grid size. \bcl In contrast to classical finite difference methods, the weighted methods enable us to use large time steps and grid sizes that are not constrained by high wave numbers in space and corresponding high frequencies in time. The key to their success lies in an exact preservation of the dispersion relation in the numerical method, which is achieved by using appropriate exponential weights in the finite-difference scheme. The weights for spatial finite differences depend on the wave number $\kappa$ appearing in the initial data, whereas the weights for temporal finite differences depend on the associated frequency, which is determined by the dispersion relation. 

In this paper we study the nontrivial transfer and adaptation of the weighted finite difference approach of \cite{shi2025weighted} to the Klein--Gordon equation \eqref{eq:KG} with initial data \eqref{eq:init}. 
A challenging new feature compared to \cite{shi2025weighted} is the appearance of
two dispersion branches of the frequencies $\pm\omega=\pm \sqrt{1+\kappa^2}$. This leads us to an $\mathcal{O}(\eps)$-approximate decomposition of the solution of \eqref{eq:KG} with general initial data \eqref{eq:init} as a superposition of two counterpropagating solutions with polarized initial data, each of which is a uni-directional wave packet up to $\mathcal{O}(\eps^2)$ perturbations. The two polarized solutions, which move with the fast group velocities $\pm (\kappa/\omega) \eps^{-1}$, are computed in co-moving coordinates. We use a direct weighted finite difference discretization of the two resulting second-order differential equations in co-moving coordinates corresponding to the two opposite group velocities. The use of co-moving coordinates allows us to maintain a relatively small spatial computational interval while capturing the original solution across the entire domain. The separate approximation of polarized solutions is of major interest, not least because the time-span of their non-negligible overlap is only of size $\mathcal{O}(\eps)$. The weighted finite difference method for approximating polarized solutions is not only asymptotic-preserving as $\eps \to 0$ but also uniformly accurate for $0<\eps\le 1$, which is due to the fact the the weighted method turns into a standard co-moving finite difference scheme when the ratios of the time step and mesh size over $\eps$ tend to zero.

While the envelope of a polarized solution of the Klein--Gordon equation is, up to $\mathcal{O}(\eps)$, determined from a nonlinear Schrödinger equation, the envelope for a single-frequency solution of the semiclassical nonlinear Schrödinger equation in \cite{shi2025weighted} is, up to $\mathcal{O}(\eps)$, determined by an advection equation. These differences are reflected in the stability  analysis of the corresponding numerical methods, leading to different CFL-type conditions for the weighted leapfrog-type explicit methods. For this second-order differential equation we also require a different stability analysis for the weighted Crank--Nicolson-type implicit method.

A common feature with \cite{shi2025weighted} in the error analysis is the functional-analytic framework using the Wiener algebra, as is familiar in geometric optics, though less so in numerical analysis. This framework allows us to handle polynomial nonlinearities in a rather comfortable way, yielding maximum norm error bounds with minimal effort. As in \cite{shi2025weighted}, a key feature of the numerical method is again the exact preservation of the dispersion relation.
\ecl

\subsection{Outline}
Section~\ref{sec:exact} presents an asymptotic analytical approximation of the continuous problem for small $\eps$, which is of first order accuracy in $\eps$ for general initial data \eqref{eq:init}. 
\bcl
This is improved to second order in $\eps$ for polarized initial data, for which the wave is uni-directional. From these results we conclude that the solution of \eqref{eq:KG} with arbitrary (non-polarized) initial data \eqref{eq:init} can be approximated, to first order in $\eps$, by the superposition of two solutions with polarized initial data. The equations for the polarized solutions are then rewritten in co-moving coordinates. These are the  equations that are discretized in this paper.

In Section~\ref{sec:algo}, we describe the algorithm studied in this paper, which is applied to the Klein--Gordon equation \eqref{eq:KG} with polarized initial values, written in co-moving coordinates. 
\ecl
We introduce a weighted finite difference method that extends the classical explicit leapfrog method in time with central finite differences in space to the case of an arbitrarily small $\eps$ in \eqref{eq:KG}--\eqref{eq:init}, without any restriction of the time step size $\tau$ and the spatial grid size $h$ by the small parameter $\eps$. We further present an analogous extension of the implicit Crank--Nicolson method.

Section~\ref{sec:main} states the main results of the paper. We derive the leading term in the modulated Fourier expansion of the numerical solution and establish second order error bounds that allow time step sizes $\tau$ and mesh widths $h$ to be chosen arbitrarily large compared to $\eps$. For $h \gg \eps$, the explicit method imposes a step-size restriction of the form $\tau \le c h^2$, which is not required for the implicit method. For both methods we have $\mathcal{O}(h^2+\tau^2+\eps)$ error bounds for general initial data \eqref{eq:init}, which improve to $\mathcal{O}(h^2+\tau^2+\eps^2)$ for polarized initial data.

The results stated in Section~\ref{sec:main} are proved in Sections~\ref{sec:consistency} and~\ref{sec:stability}. In Section~\ref{sec:consistency}, we analyze the consistency error, i.e., the defect arising when inserting a function with a controlled small distance to the exact solution into the numerical scheme. Section~\ref{sec:stability} presents a linear Fourier stability analysis, followed by a nonlinear stability analysis that bounds the numerical solution error in terms of the defect. Section~\ref{sec:num} presents the numerical experiments.

\subsection{Extensions}
We add a remark on direct generalizations of the numerical methods and their analysis as studied here: 
\\
-- While we consider the nonlinear Klein--Gordon equation with complex initial data \eqref{eq:init}, the numerical methods and theoretical results of this paper extend with minor modifications to the real case where the complex initial data in \eqref{eq:init} are replaced by their real parts. 
\\
-- Like in \cite{shi2025weighted}, the numerical methods and their theory can be extended to initial data that are sums of modulated exponentials with different wave numbers $\kappa_1,\dots,\kappa_m$. Such an extension is easily obtained when an $\mathcal{O}(h^2+\tau^2+\eps)$ error bound is desired. \bcl On the other hand, achieving second order in $\eps$ would require a more substantial analytical and computational effort, already in the case of a single wave number. 
\\
\bcl
-- The extension of the weighted finite difference {\it algorithm} to higher space dimensions in the Klein--Gordon equation is straightforward. Extending the error analysis to higher dimensions may be feasible along similar lines to the one-dimensional case, not least in view of the dimension-independent Wiener algebra framework. Such an extension is expected but will not be treated here.
\ecl

\section{Preparation: Dominant oscillatory terms, polarized initial values, superposition of polarized solutions, and co-moving coordinates}\label{sec:exact}
Before introducing the numerical method, we give approximation results for the solution of \eqref{eq:KG}--\eqref{eq:init}.
We distinguish between general initial data \eqref{eq:init} and polarized initial data, which yield an essentially uni-directional wave propagation. \bcl We show that a general solution is approximated, up to $\mathcal{O}(\eps)$, by the superposition of two polarized solutions. \ecl
We further present a reformulation of the equation in co-moving coordinates, on which the numerical discretization will be based.

\subsection{Solution approximation for general initial data}
\label{subsec:u-approx}

Up to an $\mathcal{O}(\eps)$ error, the following result for oscillatory initial data \eqref{eq:init} with wave number $\kappa/\eps$ yields bi-directional wave propagation having frequencies $\pm \omega/\eps^2$ with $ \omega = \sqrt{1 + \kappa^2} $ for the nonlinear Klein--Gordon equation \eqref{eq:KG} with oscillatory initial data \eqref{eq:init}, in accordance with the dispersion relation 
$\eps^2(\iu \omega/\eps^2)^2 - (\iu \kappa/\eps)^2 +1/\eps^2=0$, i.e.
$\omega^2=\kappa^2 +1$, of the linear Klein--Gordon equation.
We often refer to $\kappa$ and $\omega$ as the wave number and frequency instead of $\kappa/\eps$ and $\omega/\eps^2$, respectively.

\begin{proposition}[Dominant terms of the solution]
\label{prop} \bcl
Let $u(t,x)$ be the solution of  equation \eqref{eq:KG} with $0<\eps\ll 1$ for $0\le t \le T$ with oscillatory initial data \eqref{eq:init} having wave number $\kappa\ne 0$ and smooth initial profiles $a_0$, $b_0$ of bounded support. Then there exists $c$ independent of $\eps$ such that
\ecl
    \[
    \begin{aligned}
    \|u(t,\cdot) - \wt{u}(t,\cdot)\|_{L^{\infty}} &\leq c\, \eps, \quad 0 \leq t \leq T,\\
    \eps^2\|\partial_tu(t,\cdot) - \partial_t\wt{u}(t,\cdot)\|_{L^{\infty}} &\leq c\, \eps, \quad 0 \leq t \leq T,
    \end{aligned}
    \]
where $ \wt{u} $ has the form
\begin{equation}\label{eq:MFE-exact}
    \wt{u}(t,x) = a^{+}(t, \xi) \e^{\iu (\kappa x - \omega t/\eps)/\eps} + a^{-}(t, \eta) \e^{\iu (\kappa x + \omega t/\eps)/\eps}
\end{equation}
with frequency $ {\omega = \sqrt{1 + \kappa^2}} $, group velocity $ c_g = \partial_\kappa \omega = \kappa/\omega$ \bys having \eys $|c_g|<1$,
and with co-moving coordinates $ \xi = x - c_g t/\eps$, $\eta = x + c_g t/\eps $. The functions $ a^{+}(t, \xi) $ and $ a^{-}(t, \eta) $ are independent of $\eps$ and satisfy the following separated nonlinear Schrödinger equations:
\begin{equation}\label{eq:a+a-}
\begin{aligned}
   2\iu \omega \partial_t a^{+} &= 
    -(1-c_g^2) \, \partial^2_{\xi} a^{+} + \lambda |a^{+}|^2 a^{+},\\ a^{+}(0, \xi) &= \tfrac12 \bigl(a_0(\xi) + \iu b_0(\xi)/\omega \bigr), \\[2mm]
    -2\iu \omega \partial_t a^{-}  &=
    -(1-c_g^2) \, \partial^2_{\eta} a^{-} + \lambda |a^{-}|^2 a^{-}, \\
    a^{-}(0, \eta) &= \tfrac12 \bigl(a_0(\eta) - \iu b_0(\eta)/\omega \bigr).
\end{aligned}
\end{equation}
\end{proposition}

We observe that the equation for $a^-$ is obtained by replacing $\omega$ by $-\omega$ and consequently $c_g=\kappa/\omega$ by $-c_g$ in the equation for $a^+$. Note that the velocity behaves as $\partial_t u = \mathcal{O}(\eps^{-2})$ so that the above error bound corresponds to the relative error.

\bcl
\begin{remark}\label{rem:regularity}
    The solutions $a^+(t,\cdot)$ and $a^-(t,\cdot)$ of the nonlinear Schr\"odinger equations \eqref{eq:a+a-} (having smooth compactly supported initial data by assumption) are smooth and decay in space with a polynomial rate of arbitrarily high degree; cf.~\cite{hayashi1986solutions,hayashi1988nonlinear,nahas2009persistent}.
    This decay property is useful in the context of this paper in two ways: Analytically, it yields the desired defect bounds in the Wiener algebra norm in the proofs of Proposition~\ref{prop} and Lemma~\ref{lem:defect-wiener}. Numerically, it indicates that truncating the domain in the co-moving coordinates $\xi$ and $\eta$ can be justified (though in practice not by the available {\it a priori} estimates, but by monitoring the decay of the numerical solutions).
\end{remark}
\ecl

The proof works in the functional setting of the Wiener algebra $A(\R)$; see, e.g., \cite{katznelson76} for properties of the Wiener algebra and \cite{colin2009short} for its use in highly oscillatory nonlinear hyperbolic equations. The space $A(\R)$ consists of all complex-valued functions $f:\mathbb{R}\to\mathbb{C}$ having a Fourier transform 
$\wh f$ that is Lebesgue-integrable, i.e., $\wh f \in L^1(\mathbb{R})$. 
$A(\mathbb{R})$ is a Banach space with the norm
$$
\| f \|_{A(\mathbb{R})} = \| \widehat f \|_{L^1(\mathbb{R})} = \int_{\mathbb{R}} |\widehat f(\theta)|\, \d\theta.
$$
The pointwise product of two functions $f,g\in A(\mathbb{R})$ is bounded by 
\begin{equation}\label{eq:alge}
\| fg \|_{A(\mathbb{R})} \le \| f \|_{A(\mathbb{R})} \,\| g \|_{A(\mathbb{R})},
\end{equation}
which makes $A(\mathbb{R})$ a Banach algebra. Moreover, the $A(\R)$-norm is invariant under multiplication of the function with an exponential $\e^{\iu \mu x}$:
\begin{equation}\label{eq:A-inv}
\| f \|_{A(\mathbb{R})} = \| g \|_{A(\mathbb{R})}
\quad\text{ for 
$\ g(x) = f(x) \e^{\iu \mu x}\ $ with arbitrary $\mu\in\R$},
\end{equation}
since $\wh g(\theta)=\wh f(\theta-\mu)$.
The maximum norm of a function in $A(\mathbb{R})$ is bounded by its $A(\mathbb{R})$-norm, and conversely, the $A(\mathbb{R})$-norm is bounded by the $L^1$-norm of the function and its derivative:
    \begin{equation}\label{eq:katz}
    \| f \|_{L^\infty(\mathbb{R})} \le \| f \|_{A(\mathbb{R})} \quad\text{ and }\quad
    \| f \|_{A(\mathbb{R})} \le c_1\,\bigl(\| f \|_{L^1(\mathbb{R})}+\| f' \|_{L^1(\mathbb{R})}\bigr).
    \end{equation}

\begin{proof} \bcl (a) {\it Defect bound.} \ecl
    Rather than directly estimating the error between $ u $ and $ \wt{u} $, we introduce a higher-order approximation as in \cite{pierce1995validity}:
    \[
    \begin{aligned}
    u_{\text{MFE}}(t,x) &= \left(a^{+} + \eps b^{+}\right) \e^{\iu (\kappa x - \omega t/\eps)/\eps} + \left(a^{-} + \eps b^{-}\right) \e^{\iu (\kappa x + \omega t/\eps)/\eps} \\
    & \quad - \frac{\eps^2 \lambda}{\kappa^2 - 9 \omega^2 + 1} \left( (a^{+})^2 \overline{a^{-}} \e^{\iu (\kappa x - 3 \omega t/\eps)/\eps} + \overline{a^{+}} (a^{-})^2 \e^{\iu (\kappa x + 3 \omega t/\eps)/\eps} \right),
    \end{aligned}
    \]
    where \bys $\overline{\,\cdot\,}$ denotes the complex conjugate, \eys $a^{+} = a^+(t, \xi) $, $ a^{-} = a^-(t, \eta) $ satisfy \eqref{eq:a+a-}, and $b^+=b^+(t, \xi, \eta)$ , $b^-=b^-(t, \xi, \eta)$ satisfy the following equations
    \begin{equation}\label{eq:b+b-}
    4\iu\kappa\partial_{\eta}b^+=2\lambda|a^-|^2a^+, \quad 4\iu\kappa\partial_{\xi}b^-=2\lambda|a^+|^2a^-.
    \end{equation}
    The defect obtained on inserting $ u_{\text{MFE}} $ into \eqref{eq:KG} is 
    \begin{equation}\label{eq:defectU}
        d(t,x)=\eps^2\partial^2_t u_{\text{MFE}} - \partial_x^2 u_{\text{MFE}} + \frac{1}{\eps^2}u_{\text{MFE}} + \lambda |u_{\text{MFE}}|^2 u_{\text{MFE}}.
    \end{equation}
    Using the expression of $u_{\text{MFE}}$, we have
    \[
    \begin{aligned}
    d(t,x) =&\Bigl(-2\iu\omega\partial_t a^+ -(1-c_g^2)\partial^2_{\xi}a^+-4\iu\kappa\partial_{\eta}b^+\Bigr.
    \\
    &\hskip 2.2cm \Bigl.+\lambda(|a^+|^2a^++2|a^-|^2a^+)\Bigr)\e^{\iu (\kappa x - \omega t/\eps)/\eps}
    \\
    +&\Bigl(2\iu\omega\partial_t a^--(1-c_g^2)\partial^2_{\eta}a^--4\iu\kappa\partial_{\xi}b^-\Bigr.
    \\
    &\hskip 2.2cm \Bigl.+\lambda(|a^-|^2a^-+2|a^+|^2a^-)\Bigr)\e^{\iu (\kappa x + \omega t/\eps)/\eps}\\
    +&\left(\frac{(9\omega^2-\kappa^2-1)\lambda}{\kappa^2 - 9 \omega^2 +1}+\lambda\right)(a^{+})^2 \overline{a^{-}}\e^{\iu (\kappa x - 3\omega t/\eps)/\eps}\\
    +&\left(\frac{(9\omega^2-\kappa^2-1)\lambda}{\kappa^2 - 9 \omega^2 +1 }+\lambda\right)(a^{-})^2 \overline{a^{+}}\e^{\iu (\kappa x + 3\omega t/\eps)/\eps}+
    \bcl
    R,
    \ecl
    \end{aligned}
    \]
    \bcl
    where $R$ is a finite sum of oscillatory exponentials modulated by fast decaying smooth functions, multiplied with an extra factor $\eps^k$ with $k\ge 1$. Using \eqref{eq:A-inv} and \eqref{eq:katz}, we obtain uniformly for $t\in[0,T]$
    $$
    \| R(t,\cdot) \|_{A(\R)} \le c \eps,
    $$
    where $c$ is independent of $\eps$.
    Since $a^{\pm}$ and $b^{\pm}$ satisfy equations \eqref{eq:a+a-} and \eqref{eq:b+b-}, the first two terms in the above formula for $d(t,x)$ vanish, and the third and fourth term have a vanishing prefactor. So we obtain $d=R$ and hence, uniformly for $t\in[0,T]$,
    \[
    \|d(t,\cdot)\|_{A(\R)} = \| R(t,\cdot) \|_{A(\R)} \le c\eps.
    \]
    (b) {\it Error equation.} 
    \ecl
    Comparing equation \eqref{eq:defectU} with \eqref{eq:KG} shows that the error $e=u-u_{\text{MFE}}$ solves the evolution equation
     \[
     \eps^2\partial^2_t e-\partial_x^2 e + \frac{1}{\eps^2}e+\lambda(|u|^2u-|u_{\text{MFE}}|^2u_{\text{MFE}}) + d=0.
     \]
     We rewrite this equation as a first-order system of the form and scaling as studied in a Wiener algebra setting in \cite{colin2009short},
     \begin{align*}
     \partial_t\begin{pmatrix}
         e\\s\\z
     \end{pmatrix}
     &=\frac{1}{\eps}A(\partial_x)
     \begin{pmatrix}
         e\\s\\z
     \end{pmatrix}-\frac{1}{\eps^2}E
     \begin{pmatrix}
         e\\s\\z
     \end{pmatrix}
      -\begin{pmatrix}
         0\\\lambda(|u|^2u-|u_{\text{MFE}}|^2u_{\text{MFE}}) + d\\0
     \end{pmatrix},
     \end{align*}
     where $s=\eps^2\partial_te$, $z=\eps\partial_x e$ and
\[
A(\partial_x)=\begin{pmatrix}
        0 & 0& 0\\
        0 & 0& \partial_x\\
        0 & \partial_x& 0
    \end{pmatrix},\quad 
    E=\begin{pmatrix}
        0 &-1 & 0\\
        1 &0 & 0\\
        0 & 0 &0
    \end{pmatrix}.
\]
The initial data $e(0)$, $s(0)$, and $z(0)$ are all of order $\mathcal{O}(\eps)$ with respect to the Wiener norm. 

\bcl 
(c) {\it Stability: bounding the error in terms of the defect.} Similar arguments to the following were used in
\cite{colin2009short} for analyzing first-order hyperbolic systems using the Wiener algebra norm. For the convenience of the reader we give a self-contained proof.
We consider the linear first-order hyperbolic equation for $v=(e,s,z)^\top$ 
with an inhomogeneity $g\in C([0,T],A(\R)^3)$,
$$
\partial_t v = \frac{1}{\eps}A(\partial_x) v -\frac{1}{\eps^2}E v + g,
$$
and take the Fourier transform (evaluated at $\theta\in\R$)
$$
\partial_t \wh v = \Bigl(\frac{1}{\eps}A(i\theta) -\frac{1}{\eps^2}E \Bigr) \wh v + \wh g.
$$
The matrix on the right-hand side
is skew-Hermitian, and hence its exponential is unitary. Using the variation-of-constants formula, it follows that (with $|\cdot|$ denoting the Euclidean norm)
$$
| \wh v (t,\theta) | \le | \wh v (0,\theta) |  + \int_0^t |\wh g(t',\theta)| \, \d t',
\qquad 0\le t \le T, \ \theta\in \R,
$$
which implies (omitting the power 3 on $A(\R)$)
\begin{equation} \label{eq:lin-est}
  \| v(t,\cdot) \|_{A(\R)} \le \| v(0,\cdot) \|_{A(\R)} 
  + \int_0^t \| g(t',\cdot)\|_{A(\R)} \, \d t',\qquad 0\le t \le T.  
\end{equation}
We use this bound with (omitting the argument $t'$)
$$
\| g \|_{A(\R)} = \bigl\| \lambda(|u|^2u-|u_{\text{MFE}}|^2u_{\text{MFE}}) + d \, \bigr\|_{A(\R)}.
$$
By construction of $u_{\text{MFE}}$ and using \eqref{eq:alge} -- \eqref{eq:katz},  we have
$$
\| u_{\text{MFE}}(t,\cdot) \|_{A(\R)} \le M, \quad\ 0\le t \le T.
$$ 
As long as $\| u(t,\cdot) \|_{A(\R)} \le 2M$, we thus have by \eqref{eq:alge}, with $C=12|\lambda|M^2$,
$$
\| g \|_{A(\R)} \le C \| e \|_{A(\R)} + \| d \|_{A(\R)}.
$$
Inserting this bound into \eqref{eq:lin-est} and using the Gronwall inequality, we obtain with $C_T=\e^{CT}$,
$$
\| v(t, \cdot) \|_{A(\R)} \le C_T \| v(0, \cdot) \|_{A(\R)} + C_T \int_0^t \| d(t', \cdot) \|_{A(\R)} \, \d t' \le C_T (c_0 + cT) \eps.
$$
For sufficiently small $\eps$, this bound is smaller than $2M$, and hence the bound remains valid for $0\le t \le T$. We further note that $\| e\|_{A(\R)} \le \| v\|_{A(\R)^3}$.
This error bound in the Wiener norm then yields the same bound in the weaker maximum norm, as stated in the proposition.
\qed
\end{proof}

\begin{remark}
    Similar results can be found in \cite{Kirrmann_Schneider_Mielke_1992,pierce1995validity}, where a different scaling is considered. Notably, using the Wiener algebra norm yields a sharper bound, improving the order of accuracy by $1/2$ compared to the results in \cite{pierce1995validity}. 
\end{remark}

\subsection{Solution approximation for polarized initial data}\label{sec:polarized} 
In the special case of polarized initial data (see~\cite{colin2009short,baumstark2024polarized}), 
the solution essentially propagates in a single direction. To derive such initial data, we first rewrite equation \eqref{eq:KG} as a first-order system of the type studied in \cite{colin2009short}, with $A(\partial_x)$ and $E$ as in the proof above,
\begin{equation}\label{eq:hyp-1}
    \partial_t\begin{pmatrix}
        u\\v\\w
    \end{pmatrix}=\frac{1}{\eps}A(\partial_x)
    \begin{pmatrix}
        u\\v\\w
    \end{pmatrix}-\frac{1}{\eps^2}E\begin{pmatrix}
        u\\v\\w
    \end{pmatrix}-\begin{pmatrix}
        0\\\lambda|u|^2u\\0
    \end{pmatrix}, 
\end{equation}
by introducing $v=\eps^2\partial_t u$ and $w=\eps\partial_x u$. Consider the hermitian matrix
\[
-A(\kappa)-\iu E=\begin{pmatrix}
        0 & \iu& 0\\
        -\iu & 0& -\kappa\\
        0 & -\kappa& 0
    \end{pmatrix},
\]
which has the three eigenvalues $0,\pm\sqrt{1+\kappa^2}$. Let $\omega$ be one of the non-zero eigenvalues and 
$\nu=(1,-\iu\omega,\iu\kappa)^\top$ the corresponding eigenvector. 

Initial data 
$$
    \quad \begin{pmatrix}
        u(0,x)\\v(0,x)\\w(0,x)
    \end{pmatrix}=\begin{pmatrix}
        u_0(x)\\v_0(x)\\w_0(x)
    \end{pmatrix}
    \quad\text{ with }\quad
    u_0(x)=a_0(x)\e^{\iu\kappa x/\eps}
$$
for \eqref{eq:hyp-1} are called {\it polarized initial data of the first-order system \eqref{eq:hyp-1}} if $(u_0(x),v_0(x),w_0(x))^\top$ is in the span of the eigenvector $\nu$ for all $x\in\R$.
Corresponding to $u_0(x)=a_0(x)\e^{\iu\kappa x/\eps}$, we thus have the initial data
\[
v_0(x)=-\iu\omega a_0(x)\e^{\iu\kappa x/\eps}, \quad w_0(x)=\iu\kappa a_0(x)\e^{\iu\kappa x/\eps}.
\]
On the other hand, by definition of $u_0(x)$ we have
\[
\eps \,\partial_x u_0(x) = \iu\kappa a_0(x)\e^{\iu\kappa x/\eps} + \eps\, \partial_x a_0(x) \e^{\iu\kappa x/\eps}\neq w_0(x) 
\quad\text{ in general}.
\]
To resolve this apparent contradiction and obtain initial data that are both polarized and fully consistent, 
we use the co-moving coordinate
\[
\xi = x - c_g t / \eps, \quad p(t,\xi) = u(t,x),
\]
and consider the equation for the transformed variable $p$,
\begin{equation}\label{eq:p}
\eps^2 \partial^2_t p - (1-c_g^2) \partial^2_{\xi} p - 2 \eps c_g \partial_t \partial_{\xi} p + \frac{1}{\eps^2} p + \lambda |p|^2 p = 0, 
\end{equation}
This second-order system is equivalent to the original equation \eqref{eq:KG}. 

Corresponding to the polarized initial data $(u_0,v_0,w_0)^\top$  of the first-order system \eqref{eq:hyp-1}, we define {\it polarized initial data of the second-order system \eqref{eq:p}} as
\begin{equation}\label{eq:pola}
p(0,\xi)=u_0(\xi), \quad\  
\eps^2\partial_t p(0,\xi) = v_0(\xi) +  c_g w_0(\xi)=(-\iu\omega+\iu\kappa c_g)p(0,\xi).
\end{equation}
Since we have the relation
\[
\partial_t p(t,\xi) = \partial_t u(t,x) + \frac{c_g}{\eps} \partial_x u(t,x),
\]
we recover a consistent initial time derivative for $u$ by
\[
\begin{aligned}
\eps^2\partial_t u(0,x) &= \eps^2\partial_t p(0,x) - \eps c_g \partial_x u_0(x)\\
&=v_0(x)+ c_g(w_0(x)-\eps\partial_x u_0(x))\\
&=-\iu\omega a_0(x)\e^{\iu\kappa x/\eps}-\eps c_g\partial_x a_0(x)\e^{\iu\kappa x/\eps},
\end{aligned}
\]
where we use the above formula for $\partial_x u_0(x)$. We thus derive the initial data 
\begin{equation}\label{eq:init-pola}
u(0,x) = a_0(x)\e^{\iu \kappa  x / \eps},\quad \partial_t u(0,x) = -\frac{1}{\eps^2}(\iu\omega a_0(x)+\eps c_g\partial_x a_0(x))\e^{\iu \kappa  x / \eps},
\end{equation}
which we call {\it polarized initial data of the Klein--Gordon equation \eqref{eq:KG}}.

\begin{proposition}[Dominant term in the case of polarized initial data]
\label{prop2} Let $u(t,x)$ be the solution of  equation \eqref{eq:KG} with polarized initial data \eqref{eq:init-pola}. Then there exists a positive constant $c$ such that
    \[
    \begin{aligned}
    \|u(t,\cdot) - \wt{u}(t,\cdot)\|_{L^{\infty}} &\leq c\, \eps^2, \quad 0 \leq t \leq T,\\
    \eps^2\|\partial_tu(t,\cdot) - \partial_t\wt{u}(t,\cdot)\|_{L^{\infty}} &\leq c\, \eps^2, \quad 0 \leq t \leq T,
    \end{aligned}
    \]
where $ \wt{u} $ has the form
\[
    \wt{u}(t,x) = (a(t, \xi)+\eps b(t,\xi)) \e^{\iu (\kappa x - \omega t/\eps)/\eps},
\]
with $ \xi = x - c_g t/\eps$, frequency $ \omega = \sqrt{1 + \kappa^2} $, and the group velocity $ c_g = \partial_\kappa \omega = \kappa/\omega$. The functions $ a(t, \xi) $ and $ b(t, \xi) $ satisfy nonlinear and linear Schrödinger equations, respectively,
\begin{equation}\label{eq:ab}
\begin{aligned}
    2\iu \omega \partial_t a &= - (1-c_g^2) \, \partial^2_{\xi} a + \lambda |a|^2 a, \\
    a(0, \xi) &= a_0(\xi)\\[2mm]
    2\iu\omega\partial_t b &=-(1-c_g^2)\partial_\xi^2b+\lambda (2|a|^2b+a^2\bar{b})-2 c_g\partial_t\partial_\xi a, \\ b(0,\xi)&=0.
\end{aligned}
\end{equation}
\end{proposition}

\begin{proof}
    The defect obtained by inserting $\wt{u}$ into \eqref{eq:KG} is
    \[
    d(t,x) = \varepsilon^2 \partial_t^2 \wt{u} - \partial_x^2 \wt{u} + \frac{1}{\varepsilon^2} \wt{u} + \lambda |\wt{u}|^2 \wt{u}.
    \]
    Using the expression of $\wt{u}$, we compute
    \[
    \begin{aligned}
    &\e^{-\iu(\kappa x- \omega t/\eps)/\eps}\, d(t,x) 
    = \underbrace{- (1-c_g^2) \partial_\xi^2 a - 2\iu \omega \, \partial_t a + \lambda |a|^2 a}_{=0} 
    \\
    &\quad - \eps\underbrace{\left(2 c_g \, \partial_t \partial_\xi a-  (1-c_g^2) \partial_\xi^2 b- 2\iu \omega\,\partial_t b+\lambda (2|a|^2b+a^2\bar{b}\right)}_{=0} +\mathcal{O}(\varepsilon^2),
    \end{aligned}
    \]
    where we have used that all other terms cancel due to the equations for $a$ and~$b$. The remaining steps follow as in the proof of Proposition~\ref{prop}.
\qed\end{proof}

\bcl
\subsection{Approximation by superposition of two polarized solutions}
Combining Propositions \ref{prop} and \ref{prop2}, we obtain the following result on which our discretization will be based.

\begin{proposition}[Approximation of a non-polarized solution by superposition of two polarized solutions]
\label{prop3}
Under the assumptions of Proposition \ref{prop}, the solution $u$ of the Klein--Gordon equation \eqref{eq:KG} with initial values \eqref{eq:init}  can be decomposed as
$$
u(t,x) = u^+(t,x) + u^-(t,x) + r(t,x)   \quad\text{ with } \| r \|_{L^\infty([0,T]\times \R)} \le C\eps,
$$
where $u^\pm$ are solutions of \eqref{eq:KG} with polarized initial data corresponding to the frequencies 
$\pm\omega = \pm \sqrt{1+\kappa^2}$: with $a_0^\pm(x) = \tfrac12 \bigl( a_0(x) \pm \iu b_0(x)/\omega \bigr)$,
$$
u^\pm(0,x) = a_0^\pm(x)\e^{\iu \kappa  x / \eps},\quad \partial_t u^\pm(0,x) = -\frac{1}{\eps^2}(\pm\iu\omega a_0^\pm(x)+\eps c_g\partial_x a_0^\pm(x))\e^{\iu \kappa  x / \eps}.
$$
\end{proposition}

\begin{proof}
 We observe that the equation for $a$ in \eqref{eq:ab} (now with $a^+$ and $a^-$ in the role of $a$) is the same equation as that for $a^\pm$ in \eqref{eq:a+a-}. Moreover, we note the consistency of the initial values: 
$$
u(0,x)= a_0(x) \e^{\iu \kappa  x / \eps} = a_0^+(x)\e^{\iu \kappa  x / \eps} + a_0^-(x)\e^{\iu \kappa  x / \eps}
= u^+(0,x) + u^-(0,x)
$$
and
\begin{align*}
&\eps^2 \partial_t u(0,x)= b_0(x)\e^{\iu \kappa  x / \eps} = -\iu \omega a_0^+(x) \e^{\iu \kappa  x / \eps} + \iu \omega a_0^-(x) \e^{\iu \kappa  x / \eps} 
\\
&= \eps^2 \partial_t u^+(0,x) + \eps^2 \partial_t u^-(0,x) + \mathcal{O}(\eps),
\end{align*}
where the $\mathcal{O}$ notation refers to the maximum norm.
Proposition \ref{prop2} shows that 
$$
u^\pm(t,x)= a^\pm(t,\xi)\e^{\iu (\kappa x \mp \omega t/\eps)/\eps} + \mathcal{O}(\eps) \quad \text{ with }\quad \xi=x-c_g t/\eps
$$
with $a^\pm$ from \eqref{eq:a+a-}, and hence Proposition~\ref{prop} yields the result.
\qed
\end{proof}
\ecl
\subsection{Reformulation in co-moving coordinates}\label{subsec:co-moving} 

\bcl
Since the polarized wave packets $u^\pm(t,x)$ travel over \bcl $\mathcal{O}(\eps^{-1})$ distances on the real line in $\mathcal{O}(1)$ time, \ecl
working in a fixed spatial frame would require resolving rapid transport.
Instead, we follow each wave packet in its natural moving frame, in which the
solution remains spatially localized and evolves on an $\mathcal{O}(1)$ spatial
scale, making the formulation more accessible to numerical discretization.
To this end, we change to the co-moving coordinates
\[
\begin{aligned}
\xi&=x-c_gt/\eps, \quad\
\eta=x+c_gt/\eps, 
\end{aligned}
\]
\bcl
and consider the polarized solutions in the transformed variables,
\begin{equation} \label{eq:pq-def}
    p(t,\xi) = u^+(t,x), \quad\ q(t,\eta)=u^-(t,x).
\end{equation}
We obtain the following equations for functions $p(t,\xi)$ and $q(t,\eta)$ that are solutions to \eqref{eq:KG} in co-moving coordinates, equipped with polarized initial values 
\bcl as stated in Proposition~\ref{prop3}: \ecl
\begin{align}\label{eq:rescaling-p}
    &\eps^2 \partial^2_t p - (1-c_g^2) \partial^2_{\xi} p - 2 \eps c_g \partial_\xi \partial_{t} p + \frac{1}{\eps^2} p + \lambda |p|^2 p = 0,
    \\  \nonumber
    &p(0,\xi)=a^+(0,\xi)\e^{\iu\kappa\xi/\eps},\quad\partial_tp(0,\xi)=
    \frac{\iu\vartheta}{\eps^2}\,p(0,\xi),
\\[3mm] 
\label{eq:rescaling-q}
    &\eps^2 \partial^2_t q - (1-c_g^2) \partial^2_{\eta} q + 2 \eps c_g \partial_t \partial_{\eta} q + \frac{1}{\eps^2} q + \lambda |q|^2 q = 0,
    \\  \nonumber
    &q(0,\eta)=a^-(0,\eta)\e^{\iu\kappa\eta/\eps},\quad\partial_tq(0,\eta)=\frac{-\iu\vartheta}{\eps^2}\,q(0,\eta),
\end{align}
where
\begin{equation} \label{eq:vartheta}
\vartheta=\kappa c_g - \omega = -(1-c_g^2)\,\omega=-1/\omega.
\end{equation}
Since these equations have polarized initial data in the sense of
\eqref{eq:pola}, Proposition~\ref{prop2} shows that
\begin{equation}\label{eq:pq}
\begin{aligned}
p(t,\xi)&= \wt{p}(t,\xi)+\mathcal{O}(\eps)\quad \text{with}\quad\wt{p}(t,\xi)=a^+(t,\xi)\e^{\iu\kappa\xi/\eps}\e^{\iu \vartheta t/\eps^2},\\
q(t,\eta)&= \wt{q}(t,\eta)+\mathcal{O}(\eps)\quad \text{with}\quad \wt{q}(t,\eta)=a^-(t,\eta)\e^{\iu\kappa\eta/\eps}\e^{-\iu\vartheta t/\eps^2},\\
\end{aligned}
\end{equation}
where $a^{+}$ and $a^{-}$ satisfy \eqref{eq:a+a-}.
\bcl
By Proposition~\ref{prop3} and \eqref{eq:pq-def}, we have
\ecl
\begin{equation}\label{eq:upq}
u(t,x)=p(t,\xi)+q(t,\eta)+\mathcal{O}(\eps).
\end{equation}
To sum up, we approximate the solution of the nonlinear Klein--Gordon equation \eqref{eq:KG} with (non-polarized) initial data \eqref{eq:init} by the superposition of two solutions of \eqref{eq:KG} with polarized initial data that correspond to the two frequencies $\omega$ and $-\omega$.  These two solutions are determined in co-moving coordinates as solutions of \eqref{eq:rescaling-p}  and \eqref{eq:rescaling-q}.

\section{Weighted finite difference methods}\label{sec:algo}

In this section we derive and formulate the numerical methods that are proposed and analyzed in this paper.
\bcl The method discretizes equations \eqref{eq:rescaling-p} and \eqref{eq:rescaling-q} with polarized initial values and combines their numerical solutions according to the superposition of polarized solutions \eqref{eq:upq}. In the case of particular interest where the given initial values are already polarized, only one of \eqref{eq:rescaling-p} or \eqref{eq:rescaling-q} needs to be solved, depending on the chosen sign of the frequency.
The numerical scheme requires no explicit knowledge of the equations of the envelope dynamics for the formulation of the algorithm, though we use these equations, such as \eqref{eq:ab}, for the error analysis.

\ecl

\subsection{Exponentially weighted finite differences}
\bcl
We already know from \eqref{eq:pq} that the solution
$p(t,\xi)$ of \eqref{eq:rescaling-p} is $\mathcal{O}(\eps)$ close to the modulated exponential
$ \wt p(t,\xi)=a^+(t,\xi)\,\e^{\iu\kappa\xi/\eps}\e^{\iu\vartheta t/\eps^2}$ with an $\eps$-independent smooth profile function $a^+(t,\xi)$ and
again with $\vartheta$ of \eqref{eq:vartheta}. We note
\ecl
\[
\begin{aligned}
    \partial_t^2 \wt p(t,\xi)&=\left(\left(\partial_t+\frac{\iu\vartheta }{\eps^2}\right)^2 \!\!a^+(t,\xi)\right) \e^{\iu\kappa\xi/\eps}\e^{\iu\vartheta t/\eps^2},\\
    \partial_\xi^2 \wt p(t,\xi)&=\left(\left(\partial_\xi+\frac{\iu\kappa}{\eps}\right)^2 \!\!a^+(t,\xi)\right) \e^{\iu\kappa\xi/\eps}\e^{\iu\vartheta t/\eps^2},\\
     \partial_t\partial_\xi \wt p(t,\xi)&=\left(\left(\partial_t+\frac{\iu\vartheta }{\eps^2}\right)\left(\partial_\xi+\frac{\iu\kappa}{\eps}\right)a^+(t,\xi)\right) \e^{\iu\kappa\xi/\eps}\e^{\iu\vartheta t/\eps^2}.\\
\end{aligned}
\]
As in \cite[Section~2]{shi2025weighted}, we approximate the partial derivatives of the smooth profile function $a^+$ by symmetric finite differences, with
a temporal step size $\tau$ and a spatial grid size $h$, up to errors of $\mathcal{O}(\tau^2)$ and $\mathcal{O}(h^2)$
resulting from the Taylor expansion of the smooth function $a^+$ at $(t,\xi)$. \bcl We then rewrite the so obtained expressions in terms of the oscillatory function $\wt p$. \ecl
We thus approximate, with $\alpha=\vartheta\tau/\eps^2$, 
\[
\begin{aligned}
&\partial_t^2 \wt p(t,\xi)\approx\left(\frac{a^+(t+\tau,\xi)-2a^+(t,\xi)+a^+(t-\tau,\xi)}{\tau^2}\right.
\\
&\hskip 2cm\left.+2\frac{\iu\vartheta }{\eps^2}\frac{a^+(t+\tau,\xi)-a^+(t-\tau,\xi)}{2\tau}-\frac{\vartheta ^2}{\eps^4}a^+(t,\xi) \right)\e^{\iu\kappa\xi/\eps}\e^{\iu\vartheta t/\eps^2}
\\[2mm]
&=\frac{\e^{-\iu\alpha}(1+\iu\alpha)\wt p(t+\tau,\xi) - 2 \wt p(t,\xi) + \e^{\iu\alpha}(1-\iu\alpha) \wt p(t-\tau,\xi)}{\tau^2} 
-\frac{\vartheta ^2}{\eps^4}\wt p(t,\xi),
\end{aligned}
\]
and analogously for $\partial_\xi^2 \wt p(t,\xi)$ and $\partial_t\partial_\xi \wt p(t,\xi)$. Only the last term, which is dominant for small $\eps$, is not a weighted finite difference.

\subsection{Exponentially weighted leapfrog algorithm}

We use weighted finite differences to discretize equations \eqref{eq:rescaling-p} and \eqref{eq:rescaling-q}.
We only formulate the discretization of equation \eqref{eq:rescaling-p} for $p(t,\xi)$.
The discretization of \eqref{eq:rescaling-q} for $q(t,\eta)$ is
analogous, with the only modification that
$\omega$ is replaced by $-\omega$.

Let $\tau>0$ denote the time step and $h>0$ the spatial mesh size. Using
weighted finite differences, we
obtain an explicit symmetric two-step scheme as follows:
Denoting $t_n = n \tau$ and $\xi_j = j h$ for $n\in\mathbb{N}$ and $j\in\mathbb{Z}$, we compute the approximation $p^{n+1}_j$ to the solution value $p(t_{n+1}, \xi_j)$ of \eqref{eq:rescaling-p} according to the formula
\begin{equation}\label{eq:wlf}
\begin{aligned} 
  &  \eps^2 \frac{\e^{-\iu\alpha}(1+\iu\alpha)p_j^{n+1} - 2 p_j^n + \e^{\iu\alpha}(1-\iu\alpha)p_j^{n-1}}{\tau^2} \\
& - (1-c_g^2) \,\frac{\e^{-\iu\beta}(1+\iu\beta)p^n_{j+1} - 2 p_j^n + \e^{\iu\beta}(1-\iu\beta)p^n_{j-1}}{h^2} \\
&   - 2 \eps c_g \Bigg(
        \frac{\e^{-\iu\alpha}(\e^{-\iu\beta}p_{j+1}^{n+1} - \e^{\iu\beta}p_{j-1}^{n+1}) 
              - \e^{\iu\alpha}(\e^{-\iu\beta}p_{j+1}^{n-1} - \e^{\iu\beta}p_{j-1}^{n-1})}{4 \tau h} \\
&     \qquad\qquad   + \frac{\iu \beta (\e^{-\iu\alpha}p^{n+1}_j - \e^{\iu\alpha}p^{n-1}_j)}{2 \tau h} +\frac{\iu\alpha (\e^{-\iu\beta}p_{j+1}^n-\e^{\iu\beta}p_{j-1}^n)}{2\tau h}
    \Bigg) \\
&    + \lambda |p_j^n|^2 p_j^n = 0,
\end{aligned}
\end{equation}
where
\begin{equation}\label{eq:alpha-beta}
\alpha = \frac{\vartheta\tau}{\eps^2} = -(1 - c_g^2)\,\frac{\omega\tau}{\eps^2}= - \frac{\tau}{\omega\eps^2}, \qquad
\beta = \frac{\kappa h}{\eps}.
\end{equation}

Note that the terms in \eqref{eq:wlf} correspond to those of \eqref{eq:rescaling-p}. 
The dominant $\mathcal{O}(\eps^{-2})$ terms that would
appear in \eqref{eq:wlf}
cancel due to the dispersion relation
$\omega^2 = 1 + \kappa^2$ together with $c_g=\kappa/\omega$ and $\vartheta=-1/\omega$: the factor multiplying $p^n_j/\eps^2$ equals
\begin{equation} \label{eq:disp0}
-\vartheta^2 + (1- c_g^2) \kappa^2  + 2 c_g \vartheta\kappa +1  = 0.
\end{equation}
The velocity of $p(t,\xi)$ can be approximated by
\[
\partial_t p(t_n, \xi_j) \approx 
\underbrace{\frac{\e^{-\iu\alpha} p^{n+1}_j - \e^{\iu\alpha} p^{n-1}_j}{2 \tau}}_{=:\,\nu^n_j} 
+ \frac{\iu \vartheta }{\eps^2} p^n_j.
\]
With the notation $\nu^n_j$, scheme \eqref{eq:wlf} can be rewritten in a compact 
form
\[
\begin{aligned}
   & \eps^2 \frac{\e^{-\iu\alpha} p^{n+1}_j 
    - 2 p^n_j + \e^{\iu\alpha} p^{n-1}_j}{\tau^2} 
    - (1-c_g^2) \frac{\e^{-\iu\beta} p^n_{j+1} - 2 p^n_j + \e^{\iu\beta} p^n_{j-1}}{h^2}
    \\
   & 
   - 2 \eps c_g \frac{\e^{-\iu\beta} \nu^n_{j+1} - \e^{\iu\beta} \nu^n_{j-1}}{2 h} - 2 \iu \omega \nu^n_j
   + \lambda |p^n_j|^2 p^n_j = 0.
\end{aligned}
\]
As a starting step, we initialize the scheme with a weighted explicit Euler step
\[
\begin{aligned}
& \left( \frac{2 \eps^2}{\tau} - 2 \iu \omega \right) \frac{\e^{-\iu\alpha} p^1_j - p^0_j}{\tau} - \frac{2 \eps^2}{\tau} \nu^0_j 
 - (1-c_g^2) \frac{\e^{-\iu\beta} p^0_{j+1} - 2 p^0_j + \e^{\iu\beta} p^0_{j-1}}{h^2} 
 \\
  &- 2 \eps c_g \frac{\e^{-\iu\beta} \nu^0_{j+1} - \e^{\iu\beta} \nu^0_{j-1}}{2 h} 
  - 2 \iu \omega \nu^0_j + \lambda |p^0_j|^2 p^0_j = 0,\\
  & \text{with}\quad  p^0_j=p(0,\xi_j),\quad \nu^0_j=\partial_tp(0,\xi_j)-\frac{\iu \vartheta }{\eps^2} p(0,\xi_j)=0 \ \text{ by \eqref{eq:rescaling-p}.}
\end{aligned}
\]

The weighted finite difference scheme tends to the classical leapfrog scheme in the limit $\tau/\eps^2\rightarrow0$ and $h/\eps\rightarrow0$. Our main interest here is, however, to use the weighted scheme with large ratios $\tau/\eps^2$ and $h/\eps$.

For this explicit method, we need a CFL-type stability condition:
\begin{equation}\label{eq:stability}
\tau\le\gamma|\omega|\eps h,\quad \gamma=\gamma(\beta)=\min(1/(2|\kappa|),1)\max(|\beta|,1).
\end{equation}
Equivalently,  $|\alpha|\le\gamma|\beta|/|\kappa|$. For large $\beta$ we note $\gamma\sim|\beta|=|\kappa h/\eps|$. This yields the condition $\tau\leq c\,h^2$, which is a CFL condition for the Schr\"odinger equation.

On the other hand, for small $\beta$, \eqref{eq:stability} becomes the CFL condition $\tau\leq \tilde{c}\,\eps h$ of the classical unweighted leapfrog method applied to \eqref{eq:KG}, which in our highly oscillatory situation requires in addition $\tau\ll\eps^2$ and $h\ll\eps$ to have a small consistency error.

In practice, the sequence $(p^n_j)_{j\in\mathbb{Z}}$ must be truncated to a finite-dimensional vector, typically by imposing periodic boundary conditions over a finite interval. In this paper we will not study the error due to the truncation to a finite interval. The truncation appears plausible in view of Remark~\ref{rem:regularity}. 

\subsection{Exponentially weighted Crank--Nicolson algorithm}
Using the same approach, we derive the following weighted Crank--Nicolson type scheme:
\begin{equation} \label{eq:wcn}
\begin{aligned}
&    \eps^2 \frac{\e^{-\iu\alpha}(1+\iu\alpha)p_j^{n+1} - 2 p_j^n + \e^{\iu\alpha}(1-\iu\alpha)p_j^{n-1}}{\tau^2} \\
&   - (1-c_g^2) \frac{\e^{-\iu\beta}(1+\iu\beta)\breve{p}^n_{j+1} - 2 \breve{p}_j^n + \e^{\iu\beta}(1-\iu\beta)\breve{p}^n_{j-1}}{h^2} \\
&    - 2 \eps c_g \Bigg(
        \frac{\e^{-\iu\alpha}(\e^{-\iu\beta}p_{j+1}^{n+1} - \e^{\iu\beta}p_{j-1}^{n+1}) 
              - \e^{\iu\alpha}(\e^{-\iu\beta}p_{j+1}^{n-1} - \e^{\iu\beta}p_{j-1}^{n-1})}{4 \tau h} \\
&\qquad\qquad        + \frac{\iu \beta (\e^{-\iu\alpha}p^{n+1}_j - \e^{\iu\alpha}p^{n-1}_j)}{2 \tau h}+\frac{\iu\alpha (\e^{-\iu\beta}\breve{p}_{j+1}^n-\e^{\iu\beta}\breve{p}_{j-1}^n)}{2\tau h}
    \Bigg) \\
&    + \lambda \frac{(|p_j^{n+1}|^2+|p_j^{n-1}|^2)\breve{p}_j^n}{2} = 0,
\end{aligned}
\end{equation}
with $\breve{p}_j^n=\tfrac12(\e^{-\iu\alpha}p^{n+1}_j+\e^{\iu\alpha}p_j^{n-1})$. 
This implicit method is unconditionally stable. 

\subsection{Velocity approximation}

Having obtained $p_j^n$ and $q_j^n$ from the proposed schemes, we compute the velocities as

\begin{align}\label{eq:v-pola}
\eps^2v_{+,j}^{n}&=-\iu\omega p^n_j-\eps c_g\frac{\e^{-\iu\beta} p^n_{j+1}-\e^{\iu\beta} p^n_{j-1}}{2h}+\eps^2\frac{\e^{-\iu\alpha} p^{n+1}_j - \e^{\iu\alpha} p^{n-1}_j}{2 \tau}\\ 
\eps^2v_{-,j}^{n}&=\iu\omega q^n_j+\eps c_g\frac{\e^{-\iu\beta} q^n_{j+1}-\e^{\iu\beta} q^n_{j-1}}{2h}+\eps^2\frac{\e^{\iu\alpha} q^{n+1}_j - \e^{-\iu\alpha} q^{n-1}_j}{2 \tau}.\nonumber
\end{align}
Then we can construct approximations of the solution $u$ and its time derivative $\partial_t u$ as follows.

\subsection{Reconstruction of $u$ and $\partial_t u$ from $p$ and $q$}\label{subsec:alg}
We construct an approximation of the solution $u(t,x)$ to \eqref{eq:KG} by distinguishing two cases:

\noindent{1. Well-separated wave packets:}  
When the wave packets are spatially separated, the solution is obtained by directly combining the two components.  With $s^n=c_g t_n/\eps$,
\[
\begin{aligned}
u^n(\xi_j + s^n) &= p_j^{n}, 
&\quad u^n(\eta_j - s^n) &= q_j^{n},\\
v^n(\xi_j + s^n) &=v^n_{+,j} 
&\quad v^n(\eta_j - s^n) &= v^n_{-,j}
\end{aligned}
\]

\noindent{2. Overlapping wave packets:}  
When the wave packets overlap, the solution is constructed using interpolation:
\[
\begin{aligned}
u_i^n &= \mathcal{I}(\{p_j^{n}\}, x_i) + \mathcal{I}(\{q_j^{n}\}, x_i),\\
v_i^n &=  \mathcal{I}(\{v_{+,j}^n\}, x_i) 
       +  \mathcal{I}(\{v_{-,j}^n\}, x_i),
\end{aligned}
\]
where $x_i$ denotes the global grid points, and $\mathcal{I}$ is an interpolation operator ensuring a smooth transition in the overlapping region. 

\bys
\begin{remark}\label{rem:interpolation}
With localized wave packets, the well-separated case occurs after a short time interval of length $\mathcal{O}(\eps)$ in view of the two opposed group velocities $\pm c_g/\eps$, so that interpolation is only rarely required for small $\eps$, and not at all when $\tau\gg \eps$.  When interpolation is required, 
we use cubic spline interpolation of the envelope values onto the global grid, followed by the reconstruction of position via \eqref{eq:pq}--\eqref{eq:upq} and the velocity in an analogous manner.
The error analysis below assumes exact reconstruction.
\end{remark}
\eys

\section{Main results}\label{sec:main}
\subsection{Results for general initial data}
The following result 
provides the dominant term of the modulated Fourier expansion of the numerical solution of \eqref{eq:wlf} and \eqref{eq:wcn}. 
Here we recall that $\vartheta$ is defined by \eqref{eq:vartheta}.
\begin{theorem}[Dominant terms of the numerical solution]\label{th:MFE}
    Let $ p^{n}_j $ and $q^n_j$ be the numerical solution obtained by applying the weighted leapfrog algorithm \eqref{eq:wlf} under the stability condition \eqref{eq:stability} or by the weighted Crank–Nicolson method \eqref{eq:wcn} without requiring a stability condition. \bys Let $a^\pm(t,\xi)$ be the solutions of the nonlinear Schrödinger equations \eqref{eq:a+a-} under the same assumption as in Proposition \ref{prop}. \eys 
    Then $ p^{n}_j $ and $q^n_j$ can be written as
    \begin{align*}
    p^{n}_j &= a^+(t,\xi) \,\e^{\iu\kappa\xi/\eps} \e^{\iu\vartheta t/\eps^2} + r^+(t,\xi),
    \\
    q^{n}_j &= a^-(t,\xi) \,\e^{\iu\kappa\xi/\eps} \e^{-\iu\vartheta t/\eps^2} + r^-(t,\xi),
    \end{align*}
    for $ t = n \tau $ and $ \xi = j h $,   where $ a^\pm(t,\xi) $ satisfy \eqref{eq:a+a-}, and the remainder term is bounded by
    \[
    \|r^\pm\|_{L^\infty([0,T] \times \mathbb{R})} \leq C (\tau^2 + h^2 + \eps).
    \]
    Here, $ C $ is independent of $ \eps, \tau, h $, but depends on the final time $ T $.

\end{theorem}


A further result of this paper is the following error bound for weighted finite difference methods. It follows directly from Section~\ref{subsec:alg}, the modulated Fourier expansion of the exact solution in Proposition~\ref{prop} and the modulated Fourier expansion of the numerical solution in Theorem~\ref{th:MFE}. Note that the velocity  behaves as $\partial_t u = \mathcal{O}(\eps^{-2})$ so that the error bound given below corresponds to the relative error.

\begin{theorem}[Error bound]\label{thm}
 Under the assumptions of Theorem~\ref{th:MFE}, 
we obtain the error bounds
\[
\begin{aligned}
| u^n(x_i)-u(t_n,x_i) | &= \mathcal{O}(\tau^2 + h^2 + \eps),\\
\eps^2 \, |v^n(x_i)-\partial_t u(t_n,x_i) | &= \mathcal{O}(\tau^2 + h^2 + \eps),
\end{aligned}
\]
uniformly for $t_n = n\tau \leq T$ and $x_i$. The constant symbolized by the $ \mathcal{O} $-notation is independent of $ \eps $, the time step $ \tau $, and the mesh size $ h $, but depends on the final time $T$.

\end{theorem}

In the less interesting regime where $\tau \ll \varepsilon^2$ and $h \ll \varepsilon$, the method can be applied directly to approximate the original equation by solving only $p(t,\xi)$ or $q(t,\eta)$. For instance, one can set $q \equiv 0$ (or, alternatively, $p \equiv 0$) with initial data
\begin{equation}\label{eq:initial}
\begin{aligned}
p(0,\xi) &= u(0,\xi), &\quad \partial_t p(0,\xi) &= \partial_t u(0,\xi) + \frac{c_g}{\varepsilon} \, \partial_x u(0,\xi),\\
q(0,\eta) &= 0, &\quad \partial_t q(0,\eta) &= 0.
\end{aligned}
\end{equation}
To implement, it suffices to include a conditional check in the code of setting the initial data for $p$ and $q$ as follows:

If $h^2 \ge c \, \varepsilon^5$, then we set $p^0$ and $q^0$ according to the initial data \eqref{eq:rescaling-p}--\eqref{eq:rescaling-q}.

Else we set $p^0$ and $q^0$ as \eqref{eq:initial}. 

\subsection{Polarized initial data}

In this case, the solution propagates in a single direction and it suffices to solve only one of the equations for either $p(t,\xi)$ or $q(t,\eta)$. The polarized case allows for an improved error bound. 

The following theorem states the error bound for the case $\omega = \sqrt{1+\kappa^2}$, i.e., when $u(t,x) = p(t,\xi)$ with $p$ satisfying \eqref{eq:rescaling-p}. The case $\omega = -\sqrt{1+\kappa^2}$ is analogous, with $q$ replacing $p$.

\begin{theorem}[Error bound for polarized initial data]\label{thm:polarized}
Let $p^n_j$ denote the numerical solution obtained by applying the weighted leapfrog  algorithm \eqref{eq:wlf} under the stability condition \eqref{eq:stability} or by the weighted Crank–Nicolson method \eqref{eq:wcn} without requiring a stability condition. Assume polarized initial data with a single frequency $\omega$ associated with $\kappa$, and let the assumptions of Theorem~\ref{th:MFE} hold.  Then there are the error bounds (with $x^n_j=\xi_j + c_g t_n/\eps$)
\[
\begin{aligned}
|p^n_j -u(t_n,x^n_j) |
&= \mathcal{O}(\tau^2 + h^2 + \eps^2),\\
\eps^2|v^n_{+,j}-\partial_t u(t_n,x^n_j)|&=\mathcal{O}(\tau^2 + h^2 + \eps^2),
\end{aligned}
\]
for $t_n = n\tau \le T$, $\xi_j = j h$.
Here, the constant symbolized by the $ \mathcal{O} $-notation is independent of $\tau$, $h$, and $0 < \eps \le 1$, but depends on the final time $T$ and the coefficient appearing in the stability condition \eqref{eq:stability}.

\end{theorem}

\bcl
\begin{remark} [$\eps$-uniform convergence] \label{rem:unif}
Since the weighted finite difference methods reduce to the classical unweighted finite difference method as $h/\eps\to 0$ and $\tau/\eps\to 0$, standard Taylor expansion yields an error bound of
$\mathcal{O}({\tau^2}/{\varepsilon^6} + {h^2}/{\varepsilon^4})$. Consequently, the error satisfies
\[
|u^n(x_i)-u(t_n,x_i)| 
\le \min\left( C_0 (\tau^2 + h^2 + \varepsilon^2), \; C_1 \Big(\frac{\tau^2}{\varepsilon^6} + \frac{h^2}{\varepsilon^4}\Big) \right),
\]
uniformly for $ t_n$, $x_i$, and $ 0 < \eps \leq 1 $. Maximizing this bound over $ 0 < \eps \leq 1 $ gives the optimal balance $ \eps^8 \sim \tau^2, \eps^6\sim h^2 $, leading to a uniform accuracy of $ O\big( \tau^{1/2} + h^{2/3} \big) $ in the maximum norm for all $ 0 < \eps \leq 1 $. The optimal quadratic convergence rate is achieved when $\eps\sim 1$ or $\eps^2 \leq \tau^2+h^2$.
\end{remark}

\begin{remark}[Weighted finite differences vs.~discretization of the envelope equations]
As an alternative numerical approach, one might consider discretizing the $\eps$-independent system of coupled nonlinear and linear Schr\"odinger equations \eqref{eq:ab} for the envelope function $a+\eps b$ by a standard finite difference method. While this appears to be a feasible approach to obtain a method with an $\mathcal{O}(\tau^2 + h^2 + \eps^2)$ error bound, we emphasize that this envelope approach is \emph{not} equivalent to the weighted finite difference method for $p(t,\xi)$. The weighted finite difference method only requires knowledge of the dispersion relation, of the Klein--Gordon equation in co-moving coordinates and of the polarized initial data, but no explicit knowledge of the system of envelope equations is required.
Moreover, the envelope discretization based on \eqref{eq:ab} is of an asymptotic nature for $\eps\to 0$. On the other hand, the weighted finite difference scheme for $p$ tends to the classical finite difference discretization of the Klein--Gordon equation in co-moving coordinates as $h/\eps \to 0$ and $\tau/\eps^2 \to 0$. As shown in Remark~\ref{rem:unif}, this yields $\eps$-uniform convergence for $0<\eps\le 1$.
\end{remark}
\ecl

\section{Consistency}\label{sec:consistency}
With the solution $a^+(t,\xi)$ of the first nonlinear Schr\"odinger  initial value problem in \eqref{eq:a+a-}, 
we consider the defect obtained on inserting 
\[
\wt{p}(t,\xi)=a^+(t,\xi)\,\e^{\iu\kappa\xi/\eps}\e^{\iu\vartheta t/ \eps^2}
\]
into the weighted finite difference scheme \eqref{eq:wlf}, 
\begin{equation}\label{eq:defect}
\begin{aligned} 
d(t,\xi) := &\eps^2 \frac{\e^{-\iu\alpha}(1+\iu\alpha)\wt{p}(t+\tau,\xi) - 2 \wt{p}(t,\xi)+ \e^{\iu\alpha}(1-\iu\alpha)\wt{p}(t-\tau,\xi)}{\tau^2 } \\
& - (1-c_g^2) \frac{\e^{-\iu\beta}\wt{p}(t,\xi+h) - 2 \wt{p}(t,\xi) + \e^{\iu\beta}\wt{p}(t,\xi-h)}{h^2} \\
    &- 2 \eps c_g \left(\frac{\e^{-\iu\alpha}(\e^{-\iu\beta}\wt{p}(t+\tau,\xi+h) - \e^{\iu\beta}\wt{p}(t+\tau,\xi-h)}{4\tau h}\right.\\
    &\qquad\qquad\left.-\frac{\e^{\iu\alpha}(\e^{-\iu\beta}\wt{p}(t-\tau,\xi+h)- \e^{\iu\beta}\wt{p}(t-\tau,\xi-h))}{4 \tau h}\right.\\
      &\qquad\qquad\left.+\frac{\iu\beta(\e^{-\iu\alpha}\wt{p}(t+\tau,\xi)-\e^{\iu\alpha}\wt{p}(t-\tau,\xi))}{2\tau h}
    \right)\\
   & + \lambda |\wt{p}(t,\xi)|^2 \wt{p}(t,\xi).
\end{aligned}
\end{equation}



\subsection{Defect bound}\label{subsec:Wiener}
We need a defect bound in the norm of the Wiener algebra
$A(\mathbb{R})$; see the summary in Section~\ref{subsec:u-approx}. 
The space $C([0,T],A(\mathbb{R}))$ in the following lemma is the Banach space of $A(\mathbb{R})$-valued continuous functions on the interval $[0,T]$, with $\| d \|_{C([0,T],A(\mathbb{R}))}=\max_{0\le t \le T} \| d(t,\cdot) \|_{A(\mathbb{R})}$.

\begin{lemma}[Defect bound in the Wiener algebra norm] \label{lem:defect-wiener}
In the situation of Theorem~\ref{th:MFE}, the defect \eqref{eq:defect} is bounded in the Wiener algebra norm by
$$
\| d \|_{C([0,T],A(\mathbb{R}))} \le c(\tau^2+h^2+\eps),
$$
where $c$ is independent of  $\eps$, $\tau$, $h$, and $n$ with $t_n=n\tau\le T$.
\end{lemma}
\begin{proof}
\bys
The proof is similar in spirit to that of the defect bound in the proof of Proposition~\ref{prop}.
We introduce the rotated defect
\[
\begin{aligned}
&\tilde d(t,\xi):=d(t,\xi)\e^{-\iu \kappa \xi / \eps} \, \e^{-\iu \vartheta  t / \eps^2}\\
&=-2\iu\omega \frac{a^+(t+\tau,\xi)-a^+(t-\tau,\xi)}{2\tau} -(1-c_g^2)\frac{a^+(t,\xi+h)-2a^+(t,\xi)+a^+(t,\xi-h)}{h^2}\\
&\quad -2\eps c_g\frac{a^+(t+\tau,\xi+h)-a^+(t-\tau,\xi+h)-(a^+(t+\tau,\xi-h)-a^+(t-\tau,\xi-h)))}{4\tau h}\\
&\quad +\eps^2\frac{a^+(t,\xi+h)-2a^+(t,\xi)+a^+(t-\tau,\xi)}{\tau^2}+\lambda|a^+|^2a^+,
\end{aligned}
\]    
which up to the $\mathcal{O}(\eps)$ terms in the two last lines (but not entirely) equals the defect of the standard finite difference method applied to \eqref{eq:a+a-}.
For the temporal finite difference we have by Taylor expansion
\[
\frac{a^+(t+\tau,\xi)-a^+(t-\tau,\xi)}{2\tau}=\partial_t a^+(t,\xi)+\tau^2 R^{(1)}_\tau(t,\xi)
\]
where
\[
R^{(1)}_\tau(t,\xi)=\int_{-1}^1\frac{1}{2}(1-|\theta|)^2\partial_t^3 a^+(t+\theta\tau,\xi)\d\theta,
\]
and similarly for the spatial finite differences with $\mathcal{O}(h^2)$ remainder terms in integral form and for the mixed finite difference with $\mathcal{O}(\tau^2+h^2)$ remainder in integral form. In view of equation \eqref{eq:a+a-} for $a^+$ and the dispersion relation \eqref{eq:disp0}, the $\mathcal{O}(1)$ terms cancel, and we obtain
\[
\tilde d(t,\xi)=-2\iu\omega \tau^2R^{(1)}_\tau(t,\xi)-(1-c_g^2)h^2R^{(2)}_h(t,\xi)-2\eps R_\eps(t,\xi).
\]
By the regularity stated in Remark~\ref{rem:regularity}, all derivatives of $a^+$ appearing in the remainder terms, as well as their spatial derivatives, are uniformly bounded in $L^1(\mathbb{R})$, hence the remainder terms satisfy
\[
R^{(1)}_\tau, R^{(2)}_h, R_\varepsilon \in C([0,T],W^{1,1}(\mathbb{R}))
\]
with bounds that are independent of $\tau$, $h$ and $\eps$.
Using \eqref{eq:A-inv} and \eqref{eq:katz} we obtain, uniformly for $t\in[0,T]$,
\[
\|d(t,\cdot)\|_{A(\mathbb{R})}=\|\tilde{d}(t,\cdot)\|_{A(\mathbb{R})}\leq 
c_1 \bigl(\|\tilde{d}(t,\cdot)\|_{L^1} + \|\partial_x\tilde{d}(t,\cdot)\|_{L^1}\bigr)
\le c(\tau^2+h^2+\eps),
\]
which is the desired bound.
\eys
\qed\end{proof}

\bcl
\subsection{Defect bound in the polarized case}\label{subsec:Wiener-pola}
For polarized initial data, the relevant defect is in fact smaller, with $\eps^2$ instead of $\eps$ in the defect bound. In this case, the defect  $d$ is defined as in \eqref{eq:defect}, but now with the $\mathcal{O}(\eps^2)$ approximation from Proposition~\eqref{prop2},
\begin{equation}\label{eq:p-pola}
\wt{p}(t,\xi) = \big(a(t,\xi) + \eps b(t,\xi)\big) \, \e^{\iu \kappa \xi / \eps} \, \e^{\iu \vartheta  t / \eps^2},
\end{equation}
where $a(t,\xi)$ and $b(t,\xi)$ are the solutions of \eqref{eq:ab}.

\begin{lemma}[Defect bound in the Wiener algebra norm, polarizedcase]\label{lem:defect-wiener-polarized}
In the situation of polarized initial data of Theorem~\ref{thm:polarized}, the defect \eqref{eq:defect} with \eqref{eq:p-pola} is bounded in the Wiener algebra norm by
$$
\| d \|_{C([0,T],A(\mathbb{R}))} \le c(\tau^2+h^2+\eps^2),
$$
where $c$ is independent of  $\eps$, $\tau$, $h$, and $n$ with $t_n=n\tau\le T$.
\end{lemma}

\begin{proof}
    The defect bound is obtained as in the previous proof, noting that now
    both the $\mathcal{O}(1)$ and $\mathcal{O}(\eps)$ terms cancel due to \eqref{eq:ab}.
\qed
\end{proof}

\subsection{Defect bounds for the Crank--Nicolson method}
The same defect bounds (possibly with different constants $c$) are obtained for the weighted Crank--Nicolson method, using the same arguments.
\ecl

\section{Stability}\label{sec:stability}


\subsection{Linear stability analysis in the Wiener algebra}
In this subsection we give linear stability results for the weighted finite difference scheme. We bound the numerical solution corresponding to the linear Klein--Gordon equation \eqref{eq:KG} (without the nonlinearity) in the Wiener algebra norm, using Fourier analysis.

We momentarily omit the nonlinearity and interpolate the weighted finite difference scheme \eqref{eq:wlf} from discrete spatial points $\xi_j=jh, j\in\mathbb{Z}$, to arbitrary $\xi\in \mathbb{R}$ by setting 
\begin{equation} \label{eq:wlf-x}
\begin{aligned}
    &\eps^2 \frac{\e^{-\iu\alpha}(1+\iu\alpha)p^{n+1}(\xi) - 2 p^n(\xi) + \e^{\iu\alpha}(1-\iu\alpha)p^{n-1}(\xi)}{\tau^2 } \\
    + &(c_g^2 - 1) \frac{\e^{-\iu\beta}p^n(\xi+h) - 2 p^n(\xi) + \e^{\iu\beta}p^n(\xi-h)}{h^2} \\
    - &2 \eps c_g \left(\frac{\e^{-\iu\alpha}(\e^{-\iu\beta}p^{n+1}(\xi+h) - \e^{\iu\beta}p^{n+1}(\xi-h))}{4 \tau h}\right.\\
    &\qquad\quad\left.- \frac{\e^{\iu\alpha}(\e^{-\iu\beta}p^{n-1}(\xi+h) - \e^{\iu\beta}p^{n-1}(\xi-h))}{4 \tau h}\right.\\
    &\qquad\quad\left.+\frac{\iu\beta(\e^{-\iu\alpha}p^{n+1}(\xi)-\e^{\iu\alpha}p^{n-1}(\xi))}{2\tau h}
    \right) = 0.
\end{aligned}
\end{equation}
We clearly have $p^n(\xi_j)=p^n_j$ of \eqref{eq:wlf} for all $n\ge 2$ if this holds true for $n=0$ and $n=1$. In particular, we have
$\max_j |p^n_j|\le \max_{\xi\in\mathbb{R}} |p^n(\xi)| \le \| p^n \|_{A(\mathbb{R})}$.

\begin{lemma}[Linear stability of the weighted leapfrog method]\label{th:stability}
    Under condition \eqref{eq:stability}, the weighted finite difference algorithm~\eqref{eq:wlf-x} without the nonlinear term is stable: There exists a norm $\vvvert\cdot\vvvert$
    on $A(\mathbb{R}) \times A(\mathbb{R})$,
    equivalent to the norm $\|\cdot\|_{A(\mathbb{R}) \times A(\mathbb{R})}$ uniformly in $\eps,\tau,h$ subject to \eqref{eq:stability}, such that  
\[
\vvvert P^{n}\vvvert = \vvvert P^{n-1}\vvvert, \qquad\text{where}\ \ 
P^{n}=\begin{pmatrix}
        p^{n+1}\\p^{n}
    \end{pmatrix}.
\]
\end{lemma}

\begin{proof}
    Let $\wh{p}^n(\theta)$ be the Fourier transform of $p^n(\xi)$, i.e., 
    \[
    p^n(\xi)=\frac{1}{2\pi}\int_\mathbb{R}\wh{p}^n(\theta)\e^{\iu\theta\xi}\d\theta.
    \]
    Substituting this into \eqref{eq:wlf-x} yields, for all $\xi$,
    \[
    \begin{aligned}
    \frac{1}{2\pi}\int_\mathbb{R}\e^{\iu \theta \xi}\Bigl(&\eps^2\frac{\e^{-\iu\alpha}(1+\iu\alpha)\wh{p}^{n+1}(\theta)-2\wh{p}^n(\theta)+\e^{\iu\alpha}(1-\iu\alpha)\wh{p}^{n-1}(\theta)}{\tau^2}\\
    &-(1-c_g^2)\frac{2(\cos(\beta-\theta h)-1)}{h^2}\wh{p}^n(\theta)\\
    &+\iu\eps c_g\left(\sin(\beta-\theta h)-\beta\right)\frac{\e^{-\iu\alpha}\wh{p}^{n+1}(\theta)-\e^{\iu\alpha}\wh{p}^{n-1}(\theta)}{\tau h}\Bigr)\d\theta=0.
    \end{aligned}
    \]
   We then have
    \[
    \begin{aligned}
   &\eps^2\frac{\e^{-\iu\alpha}(1+\iu\alpha)\wh{p}^{n+1}(\theta)-2\wh{p}^n(\theta)+\e^{\iu\alpha}(1-\iu\alpha)\wh{p}^{n-1}(\theta)}{\tau^2}\\
   &-(1-c_g^2)\frac{2(\cos(\beta-\theta h)-1)}{h^2}\wh{p}^n(\theta)\\
    &+\iu\eps c_g\left(\sin(\beta-\theta h)-\beta\right)\frac{\e^{-\iu\alpha}\wh{p}^{n+1}(\theta)-\e^{\iu\alpha}\wh{p}^{n-1}(\theta)}{\tau h}=0,
    \end{aligned}
    \]
    which is equivalent to the system
    \[
    \begin{pmatrix}
        \wh{p}^{n+1}(\theta)\\
        \wh{p}^{n}(\theta)
    \end{pmatrix}
     =G(\theta)\begin{pmatrix}
         \wh{p}^{n}(\theta)\\
         \wh{p}^{n-1}(\theta)
     \end{pmatrix},
    \]
    where
    \begin{equation}\label{eq:matrix}
          G(\theta)=\begin{pmatrix}
        \frac{2\wt{\gamma}_\theta \e^{\iu\alpha}}{1+\iu \gamma_\theta } & -\frac{(1-\iu \gamma_\theta )\e^{2\iu\alpha}}{1+\iu \gamma_\theta }\\
        1 & 0
    \end{pmatrix} 
    \end{equation}
    with
    \[
    \begin{aligned}
    \gamma_\theta &=\alpha+\frac{c_g\tau}{\eps h}(\sin(\beta-\theta h)-\beta)\\
    &=\alpha+\alpha\kappa^2(1-\sin(\beta-\theta h)/\beta)\\
    \wt{\gamma}_\theta &=1+ (1-c_g^2)\frac{\tau^2}{\eps^2h^2}\left(\cos(\beta-\theta h)-1\right)\\
    &=1-\kappa^2\frac{\alpha^2}{\beta^2}(1-\cos(\beta-\theta h)).
    \end{aligned}
    \]
To obtain the simplified expressions for $\gamma_\theta $ and $\wt{\gamma}_\theta $, we have used \eqref{eq:alpha-beta}.

    Let $\lambda_\theta ^{+}, \lambda_\theta ^{-}$ be the two roots of the characteristic polynomial
    \[
    \rho_\theta (\zeta)=\e^{-\iu\alpha}(1+\iu \gamma_\theta )\zeta^2-2\wt{\gamma}_\theta \zeta+\e^{\iu\alpha}(1-\iu \gamma_\theta ),
    \]
    namely, 
    \[
    \lambda_\theta ^{\pm}=\e^{\iu\alpha}\frac{\wt{\gamma}_\theta \pm\iu\sqrt{(1+\gamma_\theta ^2)-\wt{\gamma}_\theta ^2}}{(1+\iu \gamma_\theta )}.
    \]
    By condition \eqref{eq:stability},
\[
|\alpha|
\le
\min(1/(2\kappa^2),1/{|\kappa|})
\max(|\beta|,1)\,|\beta|.
\]
    For $|\beta|\le1$, we have $\alpha^2\leq\beta^2/\kappa^2$ and thus
    $|\wt{\gamma}_\theta|\le1$. 
    For $|\beta|>1$, we have $|\alpha|\le \beta^2/(2\kappa^2)$. Moreover, if
$|\wt{\gamma}_\theta|>1$, then
\[
|\wt{\gamma}_\theta|
<
\kappa^2\frac{\alpha^2}{\beta^2}(1-\cos(\beta-\theta h))
\le
2\kappa^2\frac{\alpha^2}{\beta^2}
\le
|\alpha|
\le
|\gamma_\theta|.
\]
Therefore, $
    |\wt{\gamma}_\theta | <\max(1,|\gamma_\theta |)
    $. This implies $1+\gamma_\theta^2>\wt{\gamma}_\theta^2$ and thus $|\lambda_\theta ^{\pm}|=1.$
    
    The vectors $(\lambda_\theta ^+,1)^\top$ and $(\lambda_\theta ^-,1)^\top$ are eigenvectors of $G(\theta)$ with eigenvalue $\lambda_\theta ^+$ and $\lambda_\theta ^-$, respectively. This is because (similar for $\lambda_\theta ^-$)
    \[
    \begin{pmatrix}
        \frac{2\wt{\gamma}_\theta\e^{\iu\alpha}}{1+\iu \gamma_\theta } & -\frac{(1-\iu \gamma_\theta )\e^{2\iu\alpha}}{1+\iu \gamma_\theta }\\
        1 & 0
    \end{pmatrix}\begin{pmatrix}
        \lambda_\theta ^+\\1
    \end{pmatrix}=\begin{pmatrix}
        \frac{(2\wt{\gamma}_\theta\e^{\iu\alpha}\lambda_\theta ^+-(1-\iu \gamma_\theta )\e^{2\iu\alpha}}{1+\iu \gamma_\theta }\\ \lambda_\theta ^+
    \end{pmatrix}=\lambda_\theta ^+\begin{pmatrix}
        \lambda_\theta ^+\\1
    \end{pmatrix}.
    \]
    Therefore $G(\theta)$ is diagonalizable,
    \begin{equation}\label{eq:diagonal}
    V(\theta) ^{-1}G(\theta)V(\theta) =\Lambda(\theta) =\text{diag}\{\lambda_\theta ^+,\lambda_\theta ^-\} 
    \quad\text{with} \quad
    V(\theta) = \begin{pmatrix}
        \lambda_\theta ^+ & \lambda_\theta ^-
        \\
        1 & 1
    \end{pmatrix}
    \end{equation}
    and $\Lambda(\theta)$ is a unitary matrix. Using the transformation matrix $V(\theta) $, we have, for any vector $y\in\mathbb{C}^2$,
     \[
    |V(\theta)^{-1}G(\theta) y|_{2}=|\Lambda(\theta)  V(\theta)^{-1}y|_{2} 
    =|V(\theta)^{-1}y|_{2}.
    \]
    Therefore,
    \[
    \begin{aligned}
     \vvvert P^{n}\vvvert 
    :=&\int_\mathbb{R} \left|V(\theta)^{-1} 
    \begin{pmatrix}
        \wh{p}^{n+1}(\theta)\\\wh{p}^{n}(\theta)
    \end{pmatrix}
    \right|_2\d\theta
    =\int_\mathbb{R} \left|V(\theta)^{-1} G(\theta)
    \begin{pmatrix}
        \wh{p}^{n}(\theta)\\\wh{p}^{n-1}(\theta)
    \end{pmatrix}
    \right|_2\d\theta
    \\
    =&\int_\mathbb{R} \left|V(\theta)^{-1} 
    \begin{pmatrix}
        \wh{p}^{n}(\theta)\\\wh{p}^{n-1}(\theta)
    \end{pmatrix}
    \right|_2\d\theta
    =\vvvert P^{n-1}\vvvert.
    \end{aligned}
    \]
    Finally, we show that 
\[
\|V(\theta) \|_2\leq C_1,\quad \|V(\theta) ^{-1}\|_2\leq C_2,\quad \forall\theta\in\mathbb{R},
\]
which yields that the newly introduced norm $\vvvert\cdot\vvvert$ is equivalent to $\|\cdot\|_{A(\mathbb{R}) \times A(\mathbb{R})}$. 
    Since
    \[
    V(\theta) ^*V(\theta) =
    \begin{pmatrix}
    2 &1+\overline{\lambda_\theta ^+}\lambda_\theta ^-\\
    1+\overline{\lambda_\theta ^-}\lambda_\theta ^+ &2
    \end{pmatrix},
    \]
the eigenvalues of $V(\theta) ^*V(\theta) $ can be calculated as $2\left(1\pm\sqrt{\frac{\wt{\gamma}_\theta ^2}{(1+\gamma_\theta ^2)}}\right)$. Since $\frac{\wt{\gamma}_\theta ^2}{(1+\gamma_\theta ^2)}\leq\mu< 1$ by condition~\eqref{eq:stability}, we have for all $k$ that
\[
\begin{aligned}
\|V(\theta) \|_2&
=\sqrt{\lambda_{\text{max}}(V(\theta) ^*V(\theta) )}<2, 
\\
\|V(\theta) ^{-1}\|_2&
=1/\sqrt{\lambda_{\text{min}}(V(\theta) ^*V(\theta) )}\leq 1/\sqrt{2(1-\mu)}, 
\end{aligned}
\]
so that 
$$\tfrac12 \,\| P \|_{A(\mathbb{R}) \times A(\mathbb{R})} \le \vvvert P \vvvert \le \frac1{\sqrt{2(1-\mu)}}\,\| P\|_{A(\mathbb{R}) \times A(\mathbb{R})}
$$
for all $P\in A(\mathbb{R}) \times A(\mathbb{R})$.
\qed\end{proof}

\begin{lemma}[Linear stability of the weighted Crank--Nicolson method]
\label{lem:stability-cn}
    The weighted Crank--Nicolson algorithm~\eqref{eq:wcn} without the nonlinear term is unconditionally stable in the sense that there exists a norm $\vvvert\cdot\vvvert$
    on $A(\mathbb{R}) \times A(\mathbb{R})$,
    equivalent to the norm $\|\cdot\|_{A(\mathbb{R}) \times A(\mathbb{R})}$ uniformly in $\eps,\tau,h$, such that
\[
\vvvert P^{n}\vvvert = \vvvert P^{n-1}\vvvert, \qquad\text{where}\ \ 
P^{n}=\begin{pmatrix}
        p^{n+1}\\p^{n}
    \end{pmatrix}.
\]
\end{lemma}
\begin{proof}
    Substituting the Fourier transform of $p^n$ into ~\eqref{eq:wcn} without the nonlinear term yields
\[
\begin{aligned}
&\eps^2\frac{\e^{-\iu\alpha}(1+\iu\alpha)\wh{p}^{n+1}(\theta)-2\wh{p}^n(\theta)+\e^{\iu\alpha}(1-\iu\alpha)\wh{p}^{n-1}(\theta)}{\tau^2}\\
&\quad -(1-c_g^2)\frac{2(\cos(\beta-\theta h)-1)}{h^2}\frac{\e^{-\iu\alpha}\wh{p}^{n+1}(\theta)+\e^{\iu\alpha}\wh{p}^{n-1}(\theta)}{2}\\
&\quad +\iu\eps c_g\left(\sin(\beta-\theta h)-\beta\right)
\frac{\e^{-\iu\alpha}\wh{p}^{n+1}(\theta)-\e^{\iu\alpha}\wh{p}^{n-1}(\theta)}{\tau h}
=0,
\end{aligned}
\]
which is equivalent to the system
    \[
    \begin{pmatrix}
        \wh{p}^{n+1}(\theta)\\
        \wh{p}^{n}(\theta)
    \end{pmatrix}
     =G(\theta)\begin{pmatrix}
         \wh{p}^{n}(\theta)\\
         \wh{p}^{n-1}(\theta)
     \end{pmatrix},
    \]
where
\[
G(\theta)=\begin{pmatrix}
\frac{2 \e^{\iu\alpha}}{\gamma_\theta} & -\frac{\overline{\gamma_\theta} \e^{2\iu\alpha}}{\gamma_\theta} \\
1 & 0
\end{pmatrix},
\]
with $\gamma_\theta = 1 + \frac{(1-c_g^2)\tau^2}{\eps^2 h^2}\bigl(1-\cos(\beta-\theta h)\bigr)
+ \iu\alpha + \iu \frac{c_g \tau}{\eps h} \bigl(\sin(\beta-\theta h)-\beta\bigr)$.
Since $|\gamma_\theta|^2 > 1$,
the eigenvalues of $G(\theta)$ are
\[
\lambda_\theta^\pm = \frac{\e^{\iu\alpha}}{\gamma_\theta} \left( 1 \pm \iu\sqrt{|\gamma_\theta|^2-1} \right)
\]
and $|\lambda_\theta^\pm| = 1$ for all $\theta$. Following the same procedure as in the proof of Lemma~\ref{lem:stability} yields the desired result.
\qed\end{proof}

\subsection{Nonlinear stability}\label{subsec:nonlinear}

\begin{lemma}[Nonlinear stability of the weighted leapfrog method] \label{lem:stability}
    Let the function $\wt{p}\in C([0,T],A(\mathbb{R}))$ be arbitrary and let the corresponding defect $d$ be defined by \eqref{eq:defect}.
    Under condition \eqref{eq:stability}, the interpolated numerical solution of \eqref{eq:wlf}, interpolated to all $\xi\in\mathbb{R}$ as in \eqref{eq:wlf-x} (but now with the nonlinear term included), satisfies the bound, for $t_n=n\tau\le T$
    \begin{align*}
    &\|p^{n}-\wt{p}(t_n,\cdot)\|_{A(\mathbb{R})} 
    \\
    &\qquad \le C \Bigl(\|p^{0}-\wt{p}(0,\cdot)\|_{A(\mathbb{R})}+\|p^{1}-\wt{p}(t_1,\cdot)\|_{A(\mathbb{R})}
    + \| d \|_{C([0,T],A(\mathbb{R}))} \Bigr),
    \end{align*}
    where $C$ is independent of $\eps$, $\tau$, $h$, and $n$ with $t_n\le T$, but depends on $T$ and on upper bounds of the above term in big brackets and of the $C([0,T],A(\mathbb{R}))$ norm of $\tilde p$.
\end{lemma}
\begin{proof}  
We define the error function $e^n(\xi)=p^{n}(\xi)-\wt{p}(t_n,\xi)$, which satisfies
\begin{equation}
\begin{aligned}\label{eq:error}
    &\eps^2 \frac{\e^{-\iu\alpha}(1+\iu\alpha)e^{n+1}(\xi) - 2 e^n(\xi) + \e^{\iu\alpha}(1-\iu\alpha)e^{n-1}(\xi)}{\tau^2 } \\
    + &(c_g^2 - 1) \frac{\e^{-\iu\beta}e^n(\xi+h) - 2 e^n(\xi) + \e^{\iu\beta}e^n(\xi-h)}{h^2} \\
    - &2 \eps c_g \left(\frac{\e^{-\iu\alpha}(\e^{-\iu\beta}e^{n+1}(\xi+h) - \e^{\iu\beta}e^{n+1}(\xi-h))}{4 \tau h}\right.\\
    &\qquad\quad\left.- \frac{\e^{\iu\alpha}(\e^{-\iu\beta}e^{n-1}(\xi+h) - \e^{\iu\beta}e^{n-1}(\xi-h))}{4 \tau h}\right.\\
    &\qquad\quad\left.+\frac{\iu\beta(\e^{-\iu\alpha}e^{n+1}(\xi)-\e^{\iu\alpha}e^{n-1}(\xi))}{2\tau h}
    \right)\\
   &+\lambda\left(|p^{n}(\xi)|^2p^{n}(\xi)-|\wt{p}(t_n,\xi)|^2\wt{p}(t_n,\xi)\right)\\
    &= -d(t,\xi).
\end{aligned}
\end{equation}
The  Fourier transform of $e^n$ then satisfies
\[
    \begin{aligned}
   &\eps^2\frac{\e^{-\iu\alpha}(1+\iu\alpha){\wh {e}}^{n+1}(\theta) -2{\wh e}^n(\theta) +\e^{\iu\alpha}(1-\iu\alpha){\wh e}^{n-1}(\theta) }{\tau^2}\\
   +&-(1-c_g^2)\frac{2(\cos(\theta h)-1)}{h^2}{\wh e}^n(\theta) \\
    +&\iu\eps c_g\left(\sin(\beta-\theta h)-\beta\right)\frac{\e^{-\iu\alpha}{\wh e}^{n+1}(\theta) -\e^{\iu\alpha}{\wh e}^{n-1}(\theta) }{\tau h}+\wh{d}^n(\theta)\\
   +&\lambda\mathcal{F} \left(|p^{n}(\xi)|^2p^{n}(\xi)-|\wt{p}(t_n,\xi)|^2\wt{p}(t_n,\xi)\right)=0.
    \end{aligned}
\]
This equation can be written in the one-step formulation
\begin{align*}
\begin{pmatrix}
{\wh e}^{n+1}(\theta) \\ {\wh e}^{n}(\theta) 
\end{pmatrix} =&
G(\theta)
\begin{pmatrix}
{\wh e}^{n}(\theta) \\ {\wh e}^{n-1}(\theta) 
\end{pmatrix}
-\frac{\tau^2}{\eps^2(1+\iu \gamma_\theta )}
\begin{pmatrix}
\wh{d}^n(\theta) \\ 0
\end{pmatrix}
\\
& -\lambda\frac{\tau^2}{\eps^2(1+\iu \gamma_\theta )}
\begin{pmatrix}
\mathcal{F}\left(|p^{n}(\xi)|^2p^{n}(\xi)-|\wt{p}(t_n,\xi)|^2\wt{p}(t_n,\xi)\right)(\theta)\\ 0
\end{pmatrix}
\end{align*}
with $\gamma_\theta $ and $G(\theta)$ defined in \eqref{eq:matrix}. Note that $\left|\frac{\tau^2}{\eps^2(1+\iu \gamma_\theta )}\right|\sim \tau$ under the stability condition \eqref{eq:stability}. Introducing ${\mathcal{E}}^{n}=\begin{pmatrix}        e^{n+1}\\
        e^{n}
     \end{pmatrix}$,
 using Lemma~\ref{th:stability} and \eqref{eq:alge} for dealing with the nonlinearity, we obtain
\[
\begin{aligned}
 \vvvert \mathcal{E}^n\vvvert
\leq&
(1+c\tau)
\vvvert{\mathcal{E}}^{n-1}\vvvert
+\widetilde c\tau \|d(t_n,\cdot)\|_{A(\mathbb{R})}\\
\leq& (1+c\tau)^n
\vvvert{\mathcal{E}}^{0}\vvvert
+\widetilde c\tau\sum_{j=1}^{n}(1+c\tau)^{n-j}\|d(t_j,\cdot)\|_{A(\mathbb{R})}
\\
\leq& \exp(cn\tau)\vvvert\mathcal{E}^{0}\vvvert
+\widetilde c\tau\frac{\exp(cn\tau)-1}{c\tau}\sup_{t\in[0,T]}\|d(t,\cdot)\|_{A(\mathbb{R})},
\end{aligned}
\]
which yields the result.
\qed\end{proof}

The same nonlinear stability bound (possibly with a different constant $C$) is obtained for the weighted Crank--Nicolson method without any CFL-type condition between $\tau$ and $h$ and $\eps$, using the same arguments together with Lemma~\ref{lem:stability-cn}.

With Lemmas \ref{lem:defect-wiener} and \ref{lem:stability} at hand, we are finally in the position to prove Theorems~\ref{th:MFE}--\ref{thm:polarized}.

\medskip\noindent
{\it Proof of Theorems \ref{th:MFE} and~\ref{thm}}.
Combined with Lemma~\ref{lem:defect-wiener} (consistency), Lemma~\ref{lem:stability} (stability) establishes the dominant terms in the modulated Fourier expansion of the numerical solution as stated in Theorem~\ref{th:MFE}. Comparing this expansion with the modulated Fourier expansion of the exact solution in Proposition~\ref{prop} then yields the solution error bound stated in Theorem~\ref{thm}.

The proof of Proposition~\ref{prop} gives the velocity
\[
\eps^2 \partial_t u(t,x) = -\iu \omega\, a^+(t,\xi) \, \e^{\iu (\kappa x - \omega t/\eps)/\eps} 
+ \iu \omega\, a^-(t,\eta) \, \e^{\iu (\kappa x + \omega t/\eps)/\eps} + \mathcal{O}(\eps),
\]
which, together with Theorem~\ref{th:MFE}, leads to the velocity error estimate in Theorem~\ref{thm}. \qed

\medskip\noindent
{\it Proof of Theorem~\ref{thm:polarized}}.
The higher-order accuracy in $\eps$ for the solution $u$ follows from the polarized approximation in Proposition~\ref{prop2} together with the refined defect estimate for the numerical solution given in Lemma~\ref{lem:defect-wiener-polarized}.

For the velocity approximation, Proposition \ref{prop2} gives
\[
\eps^2\partial_t\wt{u}(t,x)=-\left(\iu\omega (a(t,\xi)+\eps b(t,\xi))+\eps c_g\partial_\xi a(t,\xi)\right)\e^{\iu(\kappa x-\omega t/\eps)/\eps}+\mathcal{O}(\eps^2).\\
\] 
Comparing this expression with \eqref{eq:v-pola} then yields the second error bound.
\qed

\section{Numerical experiments}\label{sec:num}
We consider the one-dimensional nonlinear Klein--Gordon equation
\eqref{eq:KG}
with $\lambda = 1$. The final time is set to $T = 0.5$. Numerical errors are measured at $T$ in the discrete maximum norm.

Instead of solving the equation on the whole real line, we compute the transformed variables $p(t,\xi)$ and $q(t,\eta)$ on the bounded interval $\xi, \eta \in [-L, L]$ with $L=8$, which is large enough to fully contain the localized wave packets and allow the use of periodic boundary conditions. \bys In the numerical experiments for the given initial data, we observe that the numerical solution remains below $10^{-11}$ at the boundaries throughout the simulation time interval $t \in [0,T]$. \eys

To implement Algorithm~\ref{subsec:alg}, we distinguish two cases based on the condition
\[
L - c_g t / \eps \leq -L + c_g t / \eps, \quad\text{i.e., }\ c_g t \ge L\eps.
\]
If this condition is satisfied, the two wave packets, which move with group velocities $\pm c_g/\eps$, are  separated at time $t$. Otherwise, the wave packets overlap and interpolation is required, as described in Remark \ref{rem:interpolation}. For polarized initial data, there is only one wave packet and no such distinction is needed.

\bys
In our implementation, the linear part
is solved by fast Fourier transform. The implicit system in the weighted Crank–Nicolson method is solved using a fixed-point iteration applied to the nonlinear
term, initialized with the solution from the previous time step, and iterated until
the residual is below a tolerance of $10^{-10}$. For the exponentially weighted leapfrog algorithm, we chose the time step size
    \begin{equation}\label{eq:step}
    \tau=\begin{cases}
    h^2/5, & |\beta|>1,\\
    \eps h/20, & \text{otherwise},
    \end{cases}
    \end{equation}
    so that the stability condition \eqref{eq:stability} is satisfied for all $h$ and $\eps$. For the exponentially weighted Crank--Nicolson algorithm, we set the time step size as
    \begin{equation}\label{eq:step2}
    \tau=\begin{cases}
    h/20, & |\beta|>1,\\
    \eps h/20, & \text{otherwise}.
    \end{cases}
    \end{equation} 
    The smaller time step in the regime $|\beta|\le1$, i.e., $\varepsilon^2 \ge h^2$, is used to ensure that the temporal discretization error $\tau^2/\varepsilon^6$ (cf. Remark~\ref{rem:unif}) does not dominate, so that the observed convergence is governed by the spatial discretization error.
\eys

\bys The reference solution is computed differently depending on the value of $\eps$. For $\eps \ge 10^{-2}$, we solve the original problem using a Fourier spectral discretization in space with 8000 grid points and Strang splitting in time with a time step size of $\min(10^{-4},\tau/5)$. For $\eps<10^{-2}$, directly solving the original problem becomes prohibitively expensive. In this regime, we instead solve the limit equations for $a$ and $b$ given in Proposition~\ref{prop} or Proposition~\ref{prop2}, using a Fourier spectral discretization with 6000 grid points and Strang splitting with a time step size of $10^{-4}$.
\eys

\begin{figure}[ht]
\centerline{
\includegraphics[scale=0.48]{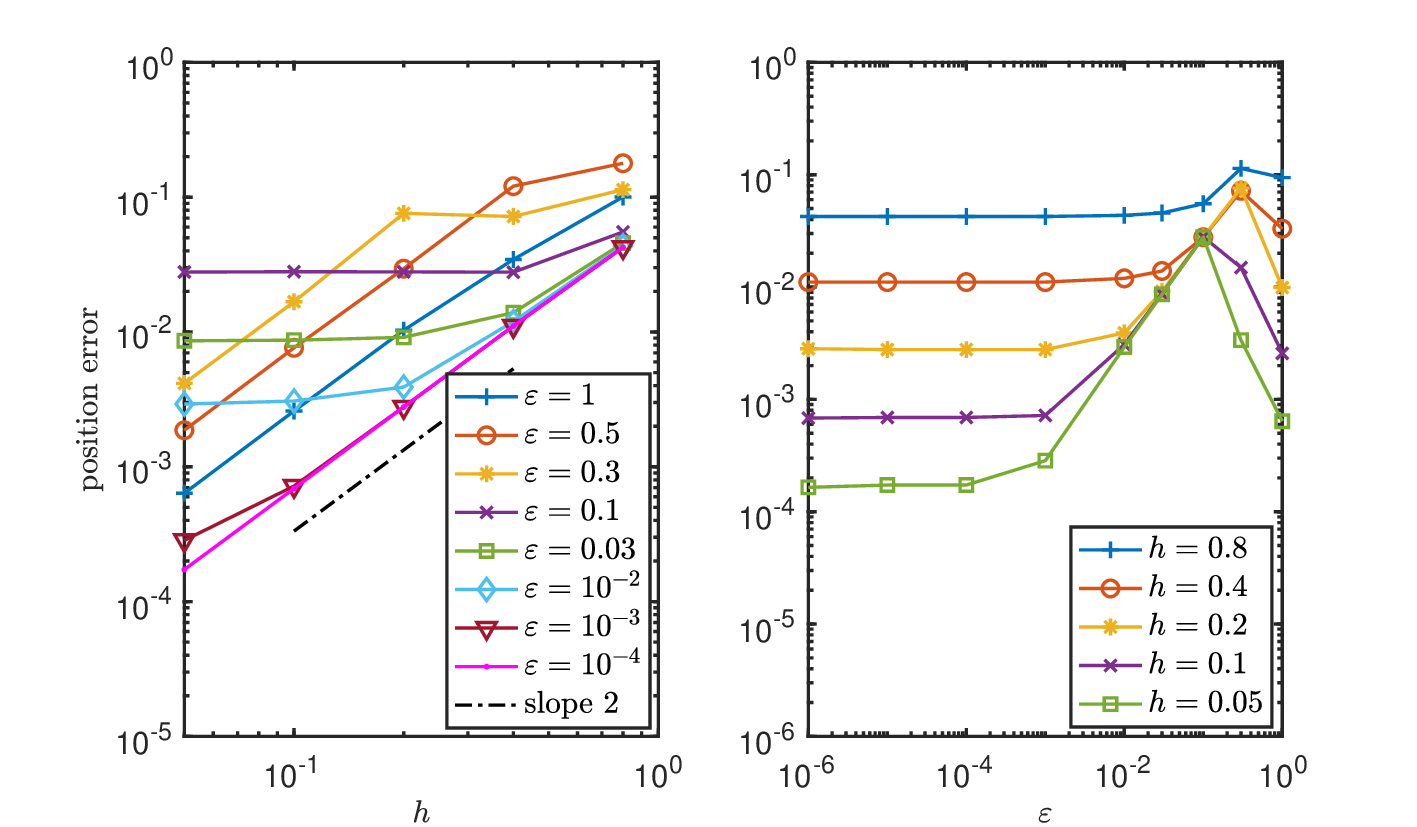}
}
\centerline{
\includegraphics[scale=0.48]{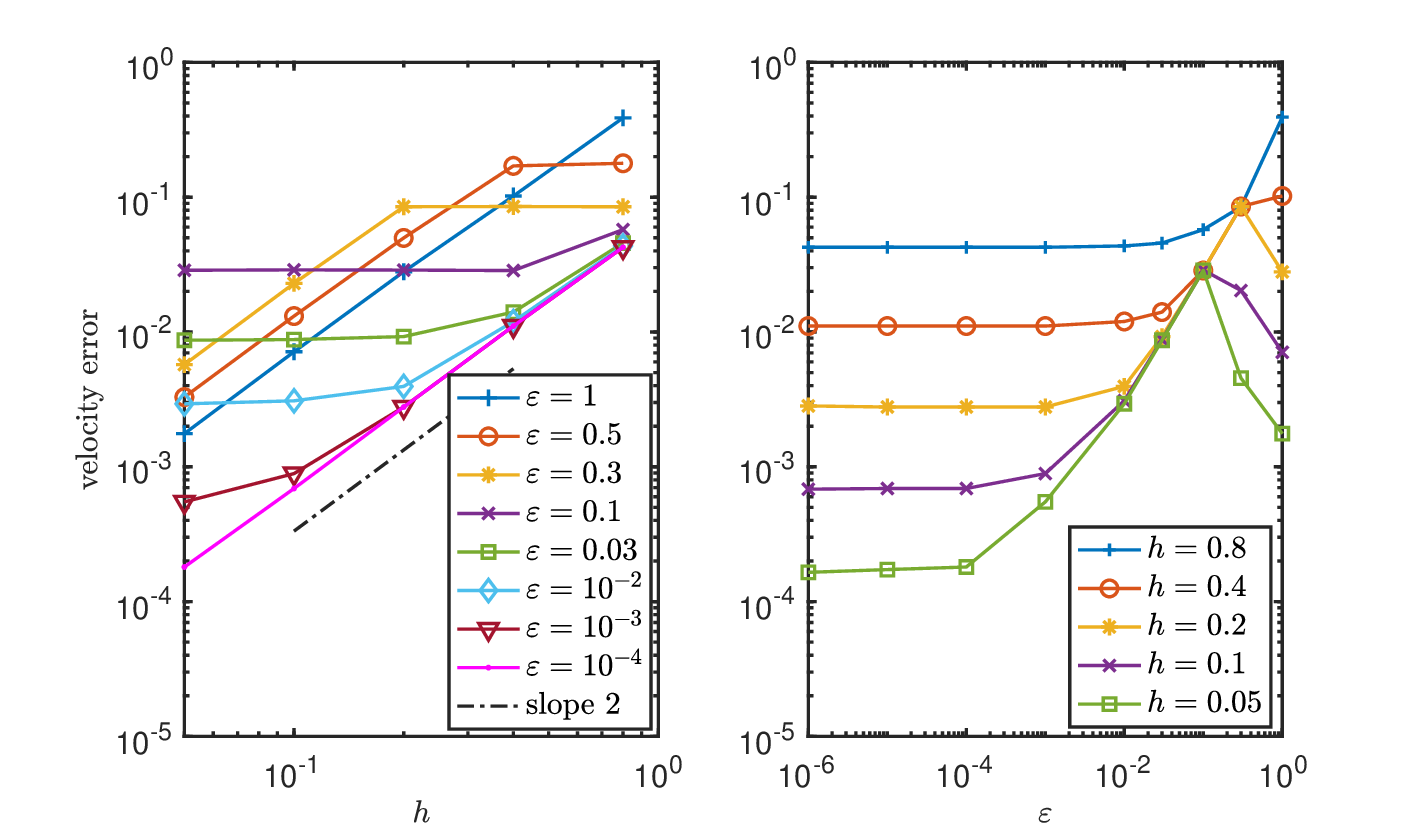}
}
\caption{ Weighted leapfrog method with general highly oscillatory initial data. Left: errors vs. $h$ for different values of $\eps$. Right: errors vs. $\eps$ for different values of $h$. }\label{fig:order}
\end{figure}

We first test our method for non-polarized initial data
\[
u(0,x) = \e^{-x^2} \e^{\iu x/\eps}, \quad 
\partial_t u(0,x) = \frac{1}{\eps^2} \e^{-x^2} \e^{\iu x/\eps}.
\]
\bys
Figure~\ref{fig:order} shows the results obtained with the weighted leapfrog method, while Figure~\ref{fig:order_CN} presents those for the weighted Crank--Nicolson method. \eys 
The left panels show the relative errors in $u$ and $\partial_t u$ versus $h$ for several fixed values of $\eps$. The error initially behaves like $h^2$ and reaches a level of $\mathcal{O}(\eps)$, consistent with the asymptotic estimate in Theorem~\ref{thm}. When $h^2 \le c \eps^5$ (with $c=5$ in this test), the $h^2$ scaling is recovered. The right panels plot the errors versus $\eps$ for fixed $h$, showing that for small $\eps$ the error levels off at a value proportional to $h^2$.

\begin{figure}[ht]
\centerline{
\includegraphics[scale=0.48]{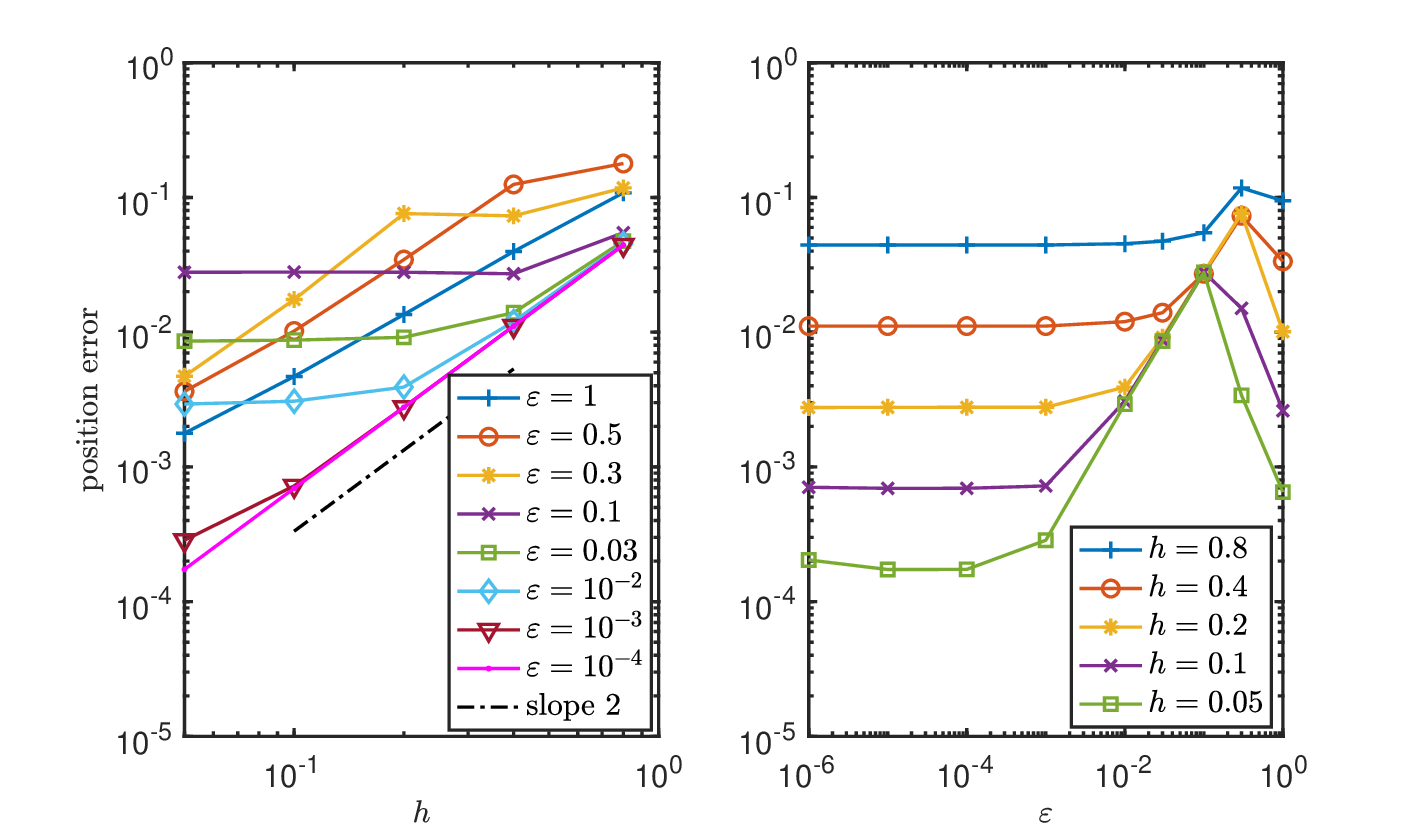}
}
\centerline{
\includegraphics[scale=0.48]{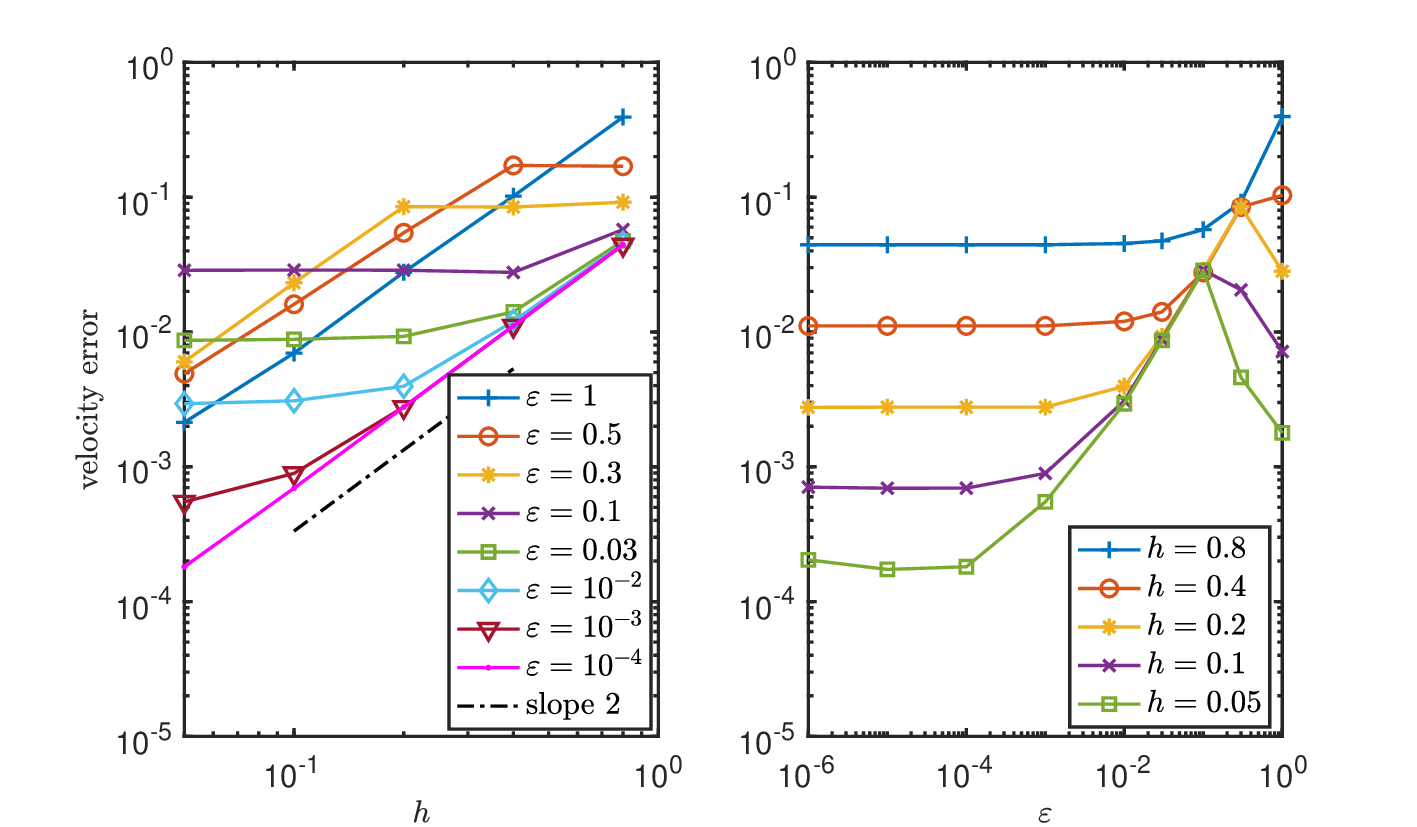}
}
\caption{ Weighted Crank--Nicoslon method with general highly oscillatory initial data. Left: errors vs. $h$ for different values of $\eps$. Right: errors vs. $\eps$ for different values of $h$. }\label{fig:order_CN}
\end{figure}

\begin{figure}[ht]
\centerline{
\includegraphics[scale=0.48]{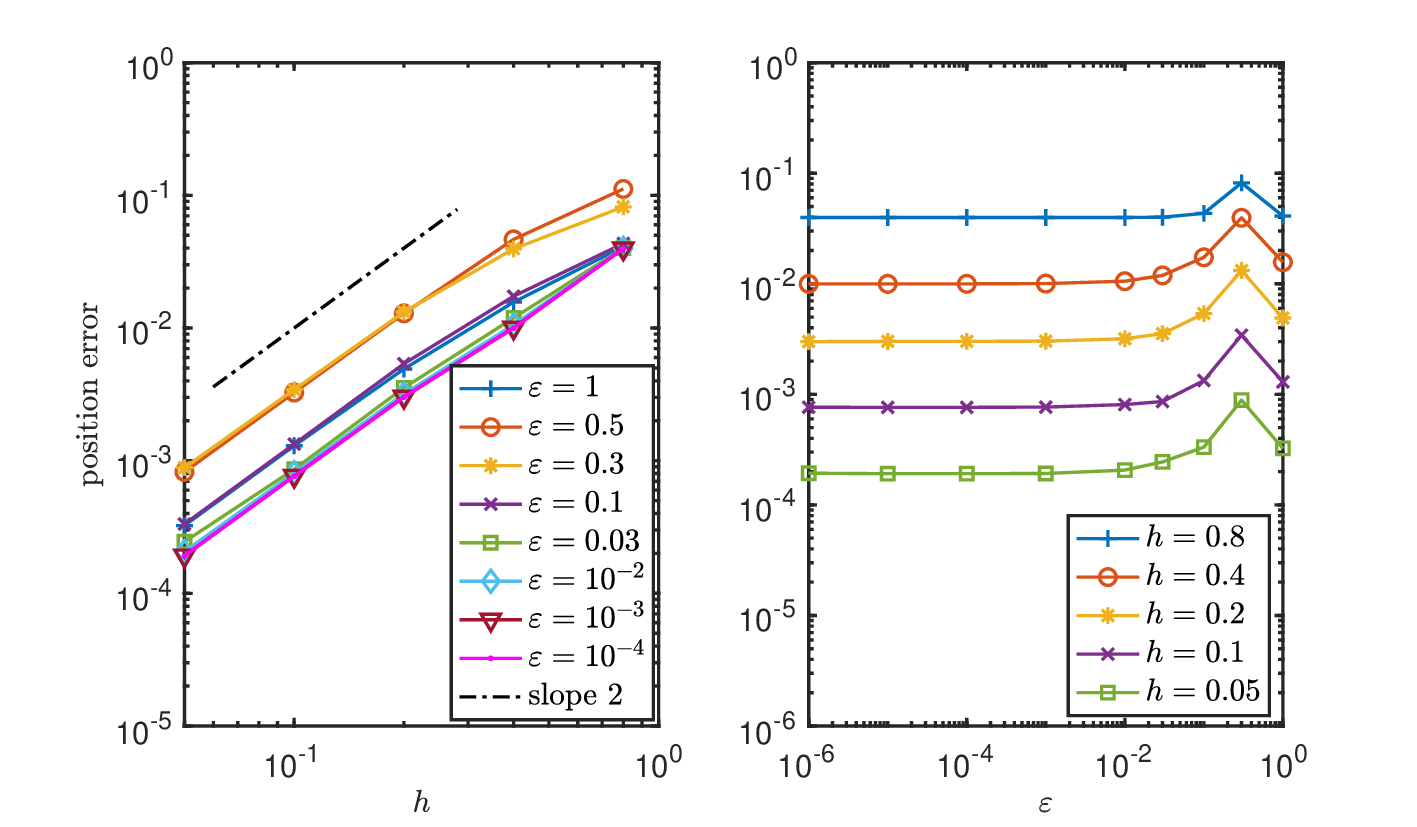}}
\centerline{
\includegraphics[scale=0.48]{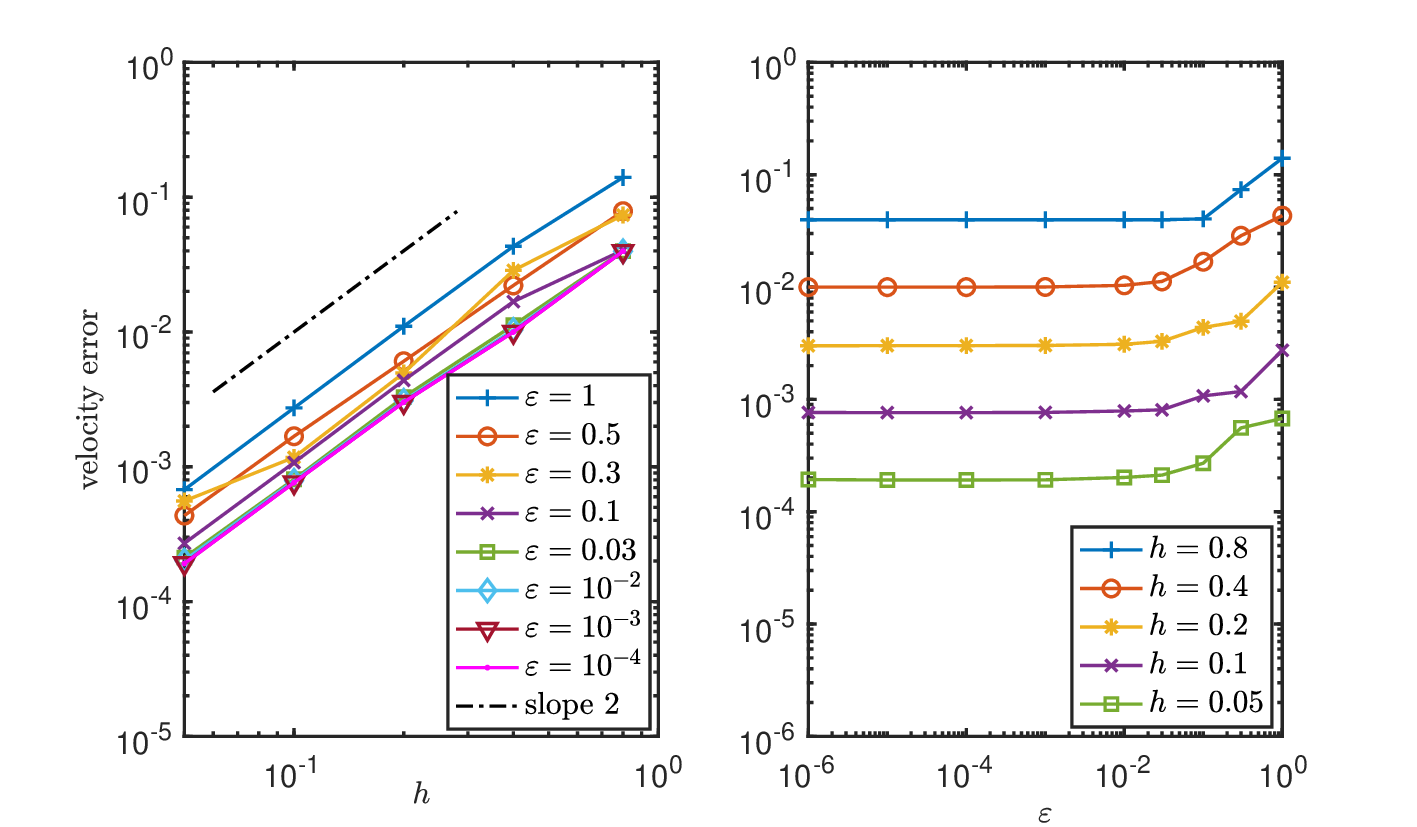}
}
\caption{ Weighted leapfrog method with polarized initial data. Left: errors vs. $h$ for different values of $\eps$. Right: errors vs. $\eps$ for different values of $h$.}\label{fig:order2}
\end{figure}

Next, we consider polarized initial data \eqref{eq:init-pola}
\[
u(0,x) = \e^{-x^2} \e^{\iu x/\eps}, \quad 
\partial_t u(0,x) = \left(-\frac{\iu \omega}{\eps^2} + \frac{2 c_g x}{\eps} \right) \e^{-x^2} \e^{\iu x/\eps}.
\]
In this case, the frequency is $\omega = \sqrt{2}$ and the group velocity is $c_g=1/\sqrt{2}$. 
\bys
Figure~\ref{fig:order2} shows the results obtained with the weighted leapfrog method, while Figure~\ref{fig:order2_CN} presents those for the weighted Crank--Nicolson method. \eys 
The left panels show the relative errors in $u$ and $\partial_tu$ versus $h$ for several fixed values of $\eps$, while the right panels show the errors plotted against $\eps$. Compared to the non-polarized case, the error is reduced for intermediate values of $\eps$, which is consistent with Theorem~\ref{thm:polarized}, where the accuracy in $\eps$ improves from $\mathcal{O}(\eps)$ to $\mathcal{O}(\eps^2)$ for polarized initial data.

\begin{figure}[ht]
\centerline{
\includegraphics[scale=0.48]{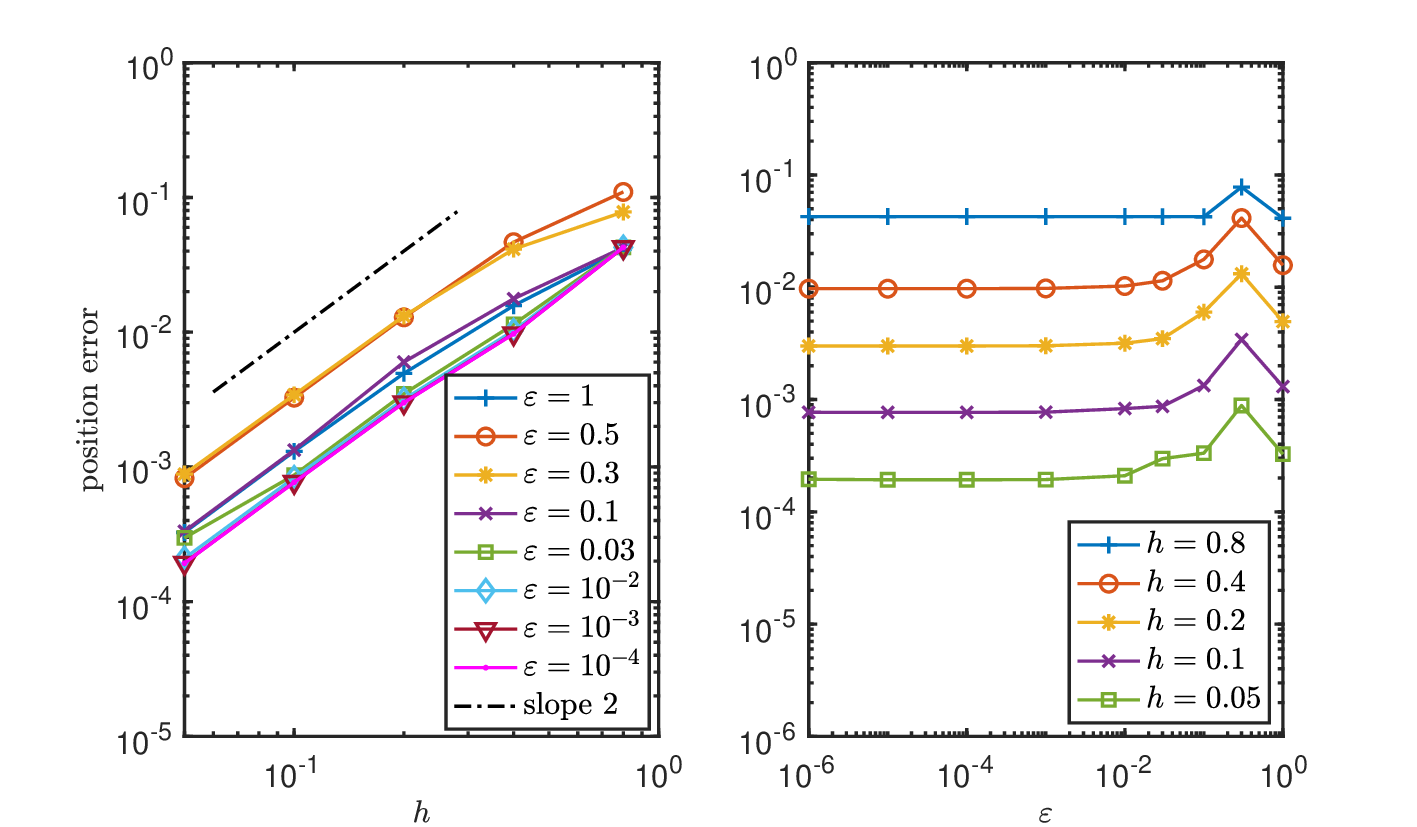}}
\centerline{
\includegraphics[scale=0.48]{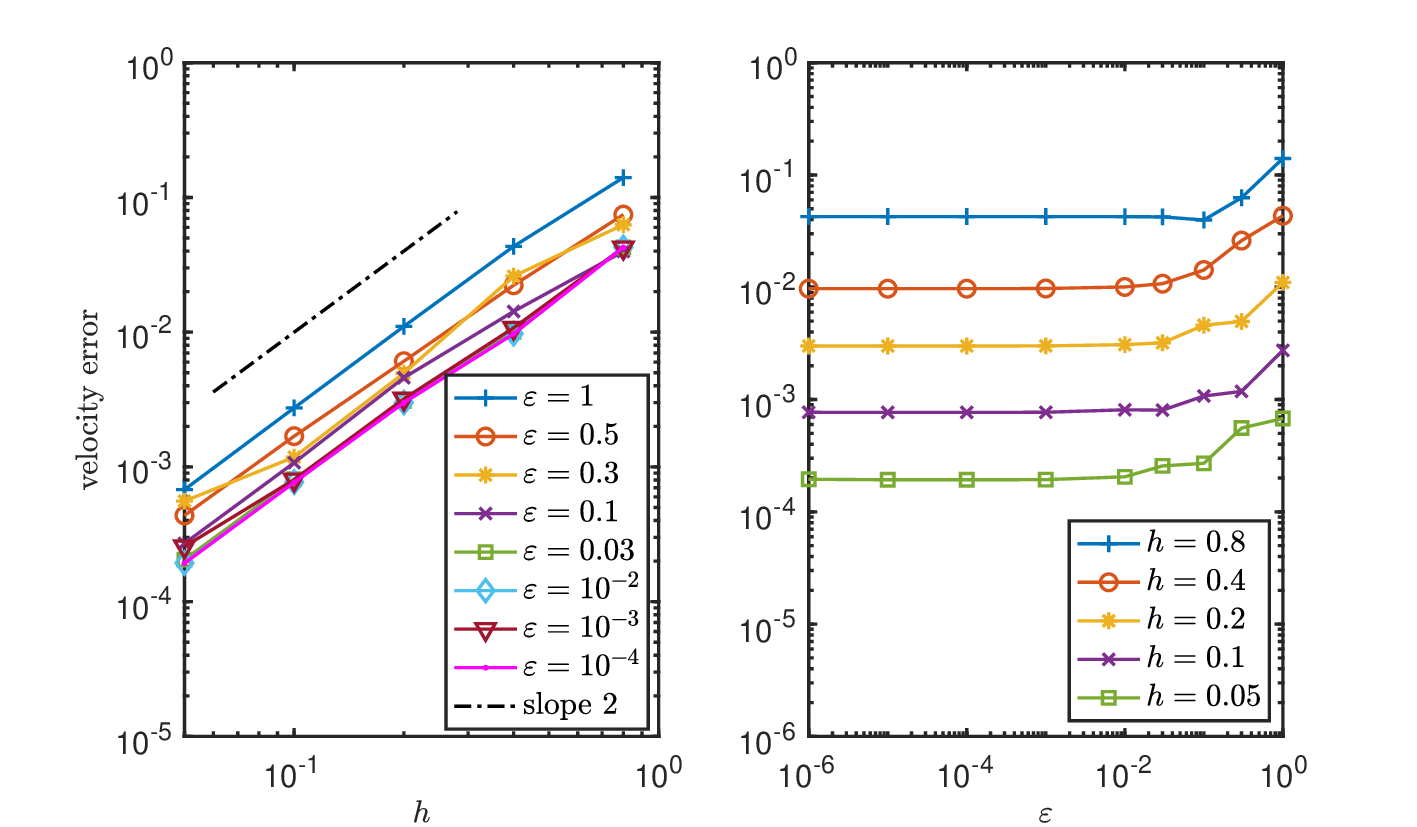}
}
\caption{ Weighted Crank--Nicolson method with polarized initial data. Left: errors vs. $h$ for different values of $\eps$. Right: errors vs. $\eps$ for different values of $h$.}\label{fig:order2_CN}
\end{figure}

\vspace{-1mm} 

\section*{Acknowledgement} 
We thank Tobias Jahnke and Johanna M\"odl for reading the first draft and providing helpful comments. 
We further thank two anonymous referees for their insightful comments on a previous version of the paper, which significantly helped and incited us to improve the presentation.
This work was supported by the Deutsche Forschungsgemeinschaft (DFG, German Research Foundation) -- Project-ID258734477 -- SFB 1173. 

\vspace{-1mm}

\bibliographystyle{abbrvurl}
\bibliography{ref}

\end{document}